\documentclass[11pt]{article}

\usepackage[a4paper,margin=28mm]{geometry}
\usepackage[T1]{fontenc}
\usepackage[utf8]{inputenc}
\usepackage{lmodern}
\usepackage{microtype}
\usepackage{amsmath,amssymb,amsthm,mathtools}
\usepackage{mathrsfs}
\usepackage{booktabs}
\usepackage{graphicx}
\usepackage{enumitem}
\usepackage{xcolor}
\usepackage{hyperref}
\usepackage[nameinlink,noabbrev]{cleveref}

\hypersetup{
  colorlinks=true,
  linkcolor=blue!55!black,
  citecolor=blue!55!black,
  urlcolor=blue!55!black,
  pdftitle={A Two-Variable Zeta Function for a Parity-Perturbed Hofstadter Q-Recursion: The Exceptional t=-1 Slice and Gaussian Boundary Layers},
  pdfauthor={Marco Mantovanelli},
  pdfsubject={Two-variable Dirichlet series for a parity-perturbed Hofstadter recursion},
  pdfkeywords={Meta-Fibonacci sequence, Mantovanelli--Hofstadter sequence, Hofstadter recursion, two-variable Dirichlet series, dyadic renormalization, log-periodic fluctuation, Mellin transform, Gaussian boundary layer, Edgeworth expansion}
}

\newtheorem{theorem}{Theorem}[section]
\newtheorem{proposition}[theorem]{Proposition}
\newtheorem{corollary}[theorem]{Corollary}
\newtheorem{lemma}[theorem]{Lemma}

\theoremstyle{definition}

\theoremstyle{remark}
\newtheorem{remark}[theorem]{Remark}
\crefname{theorem}{theorem}{theorems}
\crefname{proposition}{proposition}{propositions}
\crefname{corollary}{corollary}{corollaries}
\crefname{lemma}{lemma}{lemmas}
\crefname{conjecture}{conjecture}{conjectures}
\crefname{definition}{definition}{definitions}
\crefname{example}{example}{examples}
\crefname{remark}{remark}{remarks}

\newcommand{\Qp}{\widetilde Q}

\newcommand{\C}{\mathbb C}
\newcommand{\N}{\mathbb N}
\newcommand{\RePart}{\operatorname{Re}}

\newcommand{\Hz}{\mathfrak H}
\newcommand{\Zn}{\mathfrak Z}
\newcommand{\eps}{\varepsilon}

\title{A Two-Variable Zeta Function for a\\
Parity-Perturbed Hofstadter \texorpdfstring{$Q$}{Q}-Recursion\\[1.5mm]
\large The Exceptional \texorpdfstring{$t=-1$}{t=-1} Slice and Gaussian Boundary Layers}
\author{Marco Mantovanelli\\
\small Independent Researcher\\
\small \href{mailto:marco@mantovanelli.de}{marco@mantovanelli.de}\\
\small ORCID: \href{https://orcid.org/0009-0002-0631-293X}{0009-0002-0631-293X}}
\date{}

\begin{document}
\maketitle

\begin{abstract}
We study the two-variable Dirichlet series
\[
 Z_{\Qp}(s,t)=\sum_{n\ge1}n^{-s}\Qp(n)^{-t}
\]
attached to the parity-perturbed Hofstadter recursion.  The known estimate
$\Qp(n)=n/2+O(n/\sqrt{\log n})$ implies that its exact domain of absolute
convergence is $\RePart(s+t)>1$.  With the critical coordinate $w=s+t$, we
separate the universal term $2^t\zeta(w)$ and derive an exact transport
hierarchy, a frequency--position representation, and a dyadic
renormalization identity for the correction.

The central result concerns the exceptional slice $t=-1$.  If
$E(n)=2\Qp(n)-n$ and $A(X)=\sum_{n\le X}E(n)$, the exact binary-arch clock
gives
\[
 A(X)=X\log_2X+X\Omega\!\left(\log_2\frac{3X}{32}\right)
       +O\!\left(\frac{X}{\sqrt{\log X}}\right),
\]
where $\Omega$ is an explicit continuous periodic function.  Partial
summation continues the normalized $t=-1$ correction to $\RePart w>0$.
The Fourier series of $\Omega$ then yields a boundary resonance lattice:
a double resonance at $w=0$ and simple resonances at
$2\pi i m/\log 2$.

After subtraction of this full-slice order-$X$ skeleton, we analyze the
negative-even arch channel.  Its exact companion-forest layers have a
Gaussian limit in the weak topology against Lipschitz tests.  In particular,
a canonical negative-arch subsequence realizes Clo\^{\i}tre's optimal
$n/\sqrt{\log n}$ pointwise scale with an explicit signed constant.  The
total negative-arch mass satisfies
\[
 \mathcal A_r=\frac{512}{9\sqrt{2\pi}}\frac{16^r}{\sqrt r}
 \left(1-\frac{13}{16r}+O(r^{-2})\right).
\]
Thus the first integrated Edgeworth coefficient is $-13/16$.  These arch
results are deliberately not promoted to a full-slice continuation across
$\RePart w=0$; the remaining full-slice Mellin problem lies beyond the scope
of this paper.
\end{abstract}



\setcounter{tocdepth}{2}
\tableofcontents

\section{Introduction}
\label{sec:introduction}

Nested recurrences define their next value by using earlier values as
indices.  This self-referential mechanism is the source of both their
combinatorial richness and their analytic difficulty.  Hofstadter's
classical $Q$-recursion is the best-known example \cite{Hofstadter1979};
even basic global questions about that sequence remain open.  A broad
meta-Fibonacci literature has developed structural tools based on
recurrence families, frequency sequences, and tree models; see, for
example, \cite{CallaghanChewTanny2005,IsgurReissTanny2009}.

We study the parity-perturbed Hofstadter \(Q\)-recursion introduced by the
author in \cite{Mantovanelli2026Frequency}, whose orbit \(\Qp\) was
subsequently termed the Mantovanelli--Hofstadter sequence by Clo\^{\i}tre
\cite{Cloitre2026} and is catalogued as OEIS A394051
\cite{OEISA394051}.  It is defined by
\begin{equation}
 \Qp(1)=\Qp(2)=1,\qquad
 \Qp(n)=\Qp(n-\Qp(n-1))+\Qp(n-\Qp(n-2))+(-1)^n
 \quad(n\ge3).
 \label{eq:perturbed}
\end{equation}
Unlike the classical recursion, \eqref{eq:perturbed} is known to be
globally well-defined.  Clo\^{\i}tre proved both this fact and the
optimal-order pointwise estimate
\begin{equation}
 \Qp(n)=\frac n2+O\!\left(\frac{n}{\sqrt{\log n}}\right),
 \qquad \frac{\Qp(n)}n\longrightarrow\frac12,
 \label{eq:cloitre-asymptotic}
\end{equation}
using a binary arch-and-forest description of the orbit
\cite{Cloitre2026}.  An exact dyadic law for the multiplicities of its
values was obtained in \cite{Mantovanelli2026Frequency}.  Those two
inputs make it possible to pass from structural recursion theory to
Dirichlet-series analysis.

\subsection{The analytic object and its critical coordinate}

We attach to \eqref{eq:perturbed} the two-variable zeta function
\begin{equation}
 {
 Z_{\Qp}(s,t):=\sum_{n\ge1}\frac{1}{n^s\Qp(n)^t}.}
 \label{eq:two-variable-zeta}
\end{equation}
At $t=0$ it is the Riemann zeta function, whereas at $s=0$ it records the
value spectrum of the orbit.  The asymptotic line $\Qp(n)\sim n/2$ suggests
$Z_{\Qp}(s,t)\sim2^t\zeta(s+t)$ and singles out
\[
 w:=s+t.
\]
Writing
\[
 \lambda_n:=\log\frac{2\Qp(n)}n,
\]
the normalized correction is exactly
\begin{equation}
 {
 \Hz(w,t)=\sum_{n\ge1}\frac{e^{-t\lambda_n}-1}{n^w}.}
 \label{eq:normalized-correction-preview}
\end{equation}
Thus the two-variable problem is a Dirichlet--Laplace transform of the
ordered logarithmic slope defects.  The first part of the paper develops
three exact ways of reading this transform: differentiation along the
transport direction $\partial_t-\partial_s$, which becomes differentiation
in \(t\) at fixed \(w\), regrouping by values and their occurrence positions,
and splitting the index set into its dyadic children.  These descriptions
separate the universal scale denominator \(1-2^{1-w}\) from the genuinely
recursive remainder.

This viewpoint is close in spirit to Mellin analysis of digital sums,
where dilation produces periodic fluctuations and vertical lattices of
complex frequencies \cite{Delange1975,FlajoletGrabnerKirschenhofer1994,
FlajoletGourdonDumas1995}.  Here the dilation law is not imposed on an
external digit statistic: it is generated internally by the nested
recursion.  The resulting periodic term and its pole skeleton must
therefore be extracted from the exact arch clock.

\subsection{The exceptional linear slice}

At $t=-1$, \eqref{eq:normalized-correction-preview} becomes the ordinary
Dirichlet series
\begin{equation}
 \Hz(w,-1)=\sum_{n\ge1}\frac{E(n)}{n^{w+1}},
 \qquad E(n):=2\Qp(n)-n.
 \label{eq:intro-tminus1}
\end{equation}
The pointwise estimate \eqref{eq:cloitre-asymptotic} alone gives only a
square-root logarithmic saving.  The exact center-of-mass clock is much
stronger after summation.  We prove that
\begin{equation}
 A(X):=\sum_{n\le X}E(n)
 =X\log_2X+X\Omega\!\left(\log_2\frac{3X}{32}\right)
 +O\!\left(\frac{X}{\sqrt{\log X}}\right),
 \label{eq:intro-summatory}
\end{equation}
with an explicit continuous $1$-periodic function $\Omega$.  Partial
summation then continues $\Hz(w,-1)$ holomorphically to $\RePart w>0$,
or equivalently continues $Z_{\Qp}(s,-1)$ from $\RePart s>2$ to
$\RePart s>1$.  Fourier analysis of $\Omega$ identifies the corresponding
Abelian boundary skeleton: a double resonance at the origin and simple
dyadic resonances at $2\pi i m/\log2$.

The error left by \eqref{eq:intro-summatory} contains finer binary
geometry.  On the negative-even arches it is an exact prefix imbalance of
an ordered companion forest.  Its layer widths form a truncated even
binomial mixture, while its order inside each layer is nontrivial.  We
prove a Gaussian layer skeleton, control the integrated discrepancy from
the skeleton by exact centroid recurrences, and obtain weak convergence of
the full raw negative-arch profile against every Lipschitz test.  At the central layer
cut, the Gaussian theorem also yields an explicit even subsequence \(n_r\)
for which
\[
 \left(\frac{\Qp(n_r)}{n_r}-\frac12\right)\sqrt{\log_2n_r}
 \longrightarrow-\frac{1}{3\sqrt{2\pi}};
\]
see \cref{cor:canonical-negative-subsequence}.  This signed negative
realization complements, rather than improves, Clo\^{\i}tre's global
limsup bounds.  The exact area identity then yields
\[
 \mathcal A_r=\frac{512}{9\sqrt{2\pi}}\frac{16^r}{\sqrt r}
 \left(1-\frac{13}{16r}+O(r^{-2})\right).
\]
The coefficient $-13/16$ is an integrated Edgeworth correction: it is
deduced from exact hypergeometric corner sums and forest centroids, not
from a formal Gaussian expansion.  Appendix~\ref{subsec:second-flat-edgeworth}
pushes the same constant-test calculation one order further and obtains
$1547/1536$.

For orientation, the full-slice chain runs from
\cref{thm:centered-error-asymptotic} through
\cref{prop:closed-pole-skeleton}.  The main channel-specific conclusions are
\cref{thm:negative-gaussian-skeleton,thm:weak-gaussian-full-profile,thm:first-edgeworth-negative-area};
\cref{cor:canonical-negative-subsequence} translates the Gaussian theorem
back to the original sequence.

\subsection{Scope and organization}

Two levels of assertion must be kept distinct.  The summatory law
\eqref{eq:intro-summatory}, the continuation to $\RePart w>0$, and the
dyadic boundary resonance lattice concern the complete $t=-1$ slice.  The
Gaussian and Edgeworth theorems concern one precisely identified
negative-even arch channel.  They do not, by themselves, continue the full
slice into $\RePart w<0$.  The remaining cells of the complete Mellin
remainder are not analyzed here; their uniform treatment lies beyond the
present scope.

Sections~2--7 develop the general two-variable framework.  Section~8 treats
the exceptional slice $t=-1$, including the summatory law, continuation,
log-periodic boundary structure, and Gaussian/Edgeworth analysis.
Section~9 isolates the remaining full-slice continuation problem.
Appendix~A contains the second constant-test coefficient.  Numerical tables
and figures are explicitly labelled as checks or illustrations and are never
used as proof inputs.

\section{The two-variable zeta function}

Throughout, \(\log\) denotes the natural logarithm and
\(\log_2x=(\log x)/(\log2)\).

For positive real \(x\) and complex \(z\), we use the unambiguous convention
\(x^{-z}=e^{-z\log x}\) with the real logarithm.

\begin{theorem}[Exact absolute-convergence domain]
\label{thm:absolute-domain}
The series \eqref{eq:two-variable-zeta} converges absolutely exactly in the half-space
\begin{equation}
\boxed{\RePart(s+t)>1.}
\label{eq:absolute-halfspace}
\end{equation}
It converges normally on compact subsets of this half-space and therefore defines a jointly holomorphic function there.
\end{theorem}

\begin{proof}
By \eqref{eq:cloitre-asymptotic}, there exist constants \(0<c<C<\infty\) and \(n_0\) such that
\(cn\le\Qp(n)\le Cn\) for \(n\ge n_0\).  Hence, on any compact set of \((s,t)\)-space,
\[
\left|n^{-s}\Qp(n)^{-t}\right|
\asymp n^{-\RePart(s+t)}
\]
with constants uniform on the compact set.  Comparison with the ordinary \(p\)-series proves both convergence and divergence, and normal convergence gives holomorphy.
\end{proof}

The geometry of \eqref{eq:absolute-halfspace} already shows that \(w=s+t\) is the natural critical variable.

\section{Critical coordinates and the exact correction}

Define the normalized slope and its logarithm by
\begin{equation}
\rho_n:=\frac{2\Qp(n)}n,
\qquad
\lambda_n:=\log\rho_n.
\label{eq:rho-lambda}
\end{equation}
Then \(\rho_n\to1\), and \eqref{eq:cloitre-asymptotic} implies
\begin{equation}
\lambda_n=O\!\left((\log n)^{-1/2}\right).
\label{eq:lambda-bound}
\end{equation}
Put
\begin{equation}
\Zn(w,t):=2^{-t}Z_{\Qp}(w-t,t).
\label{eq:normalized-zeta}
\end{equation}
Then, for \(\RePart w>1\),
\begin{equation}
\Zn(w,t)
=\sum_{n\ge1}\frac{\rho_n^{-t}}{n^w}
=\zeta(w)+\Hz(w,t),
\label{eq:normalized-zeta-decomp}
\end{equation}
where
\begin{equation}
{
\Hz(w,t)
:=\sum_{n\ge1}\frac{\rho_n^{-t}-1}{n^w}
=\sum_{n\ge1}\frac{e^{-t\lambda_n}-1}{n^w}.
}
\label{eq:H-def}
\end{equation}
The correction in the original coordinates is
\begin{equation}
{
H_{\Qp}(s,t)
:=Z_{\Qp}(s,t)-2^t\zeta(s+t)
=2^t\Hz(s+t,t)
}
\qquad(\RePart(s+t)>1).
\label{eq:raw-normalized-correction}
\end{equation}

\begin{theorem}[Universal principal singularity]
\label{thm:abelian-principal}
For every compact set \(T\subset\C\), uniformly for \(t\in T\),
\begin{equation}
\lim_{\sigma\downarrow1}
(\sigma-1)\,2^{-t}Z_{\Qp}(\sigma-t,t)=1.
\label{eq:abelian-limit}
\end{equation}
Equivalently,
\begin{equation}
Z_{\Qp}(\sigma-t,t)
=\frac{2^t}{\sigma-1}
+o_T\!\left(\frac1{\sigma-1}\right).
\label{eq:principal-zeta}
\end{equation}
Moreover,
\begin{equation}
\Hz(\sigma,t)
=O_T\!\left((\sigma-1)^{-1/2}\right)
\qquad(\sigma\downarrow1).
\label{eq:H-sqrt-bound}
\end{equation}
\end{theorem}

\begin{proof}
For fixed \(t\), the coefficients \(\rho_n^{-t}\) tend to one and remain bounded; the convergence is uniform for \(t\) in a compact set.  The standard Abelian theorem for Dirichlet series \cite{Apostol1976} therefore gives \eqref{eq:abelian-limit}.

For \eqref{eq:H-sqrt-bound}, split the integers into dyadic blocks \(I_k=[2^k,2^{k+1})\).  By \eqref{eq:lambda-bound}, uniformly for \(n\in I_k\),
\[
|e^{-t\lambda_n}-1|\ll_T k^{-1/2}.
\]
If \(\sigma=1+\eps\), the contribution of \(I_k\) is therefore
\[
\ll_T 2^k\,2^{-k(1+\eps)}k^{-1/2}
=2^{-k\eps}k^{-1/2}.
\]
Summation over \(k\) gives \(O_T(\eps^{-1/2})\).
\end{proof}

The bound \eqref{eq:H-sqrt-bound} is an absolute-value estimate and does not
exploit ordered cancellation.  The exceptional slice studied below shows that
such an estimate can be far from sharp.

\section{A transport equation and the logarithmic defect hierarchy}

The baseline \(2^t\zeta(s+t)\) is characterized by a first-order transport equation.  Define
\begin{equation}
\mathcal L
:=\partial_t-\partial_s-(\log2),
\label{eq:transport-op}
\end{equation}
where the last term means multiplication by \(-\log2\).

\begin{theorem}[Exact transport hierarchy]
\label{thm:transport}
For \(\RePart(s+t)>1\),
\begin{equation}
\mathcal L\bigl(2^t\zeta(s+t)\bigr)=0,
\label{eq:baseline-transport}
\end{equation}
and for every integer \(m\ge1\),
\begin{equation}
{
\mathcal L^m Z_{\Qp}(s,t)
=(-1)^m
\sum_{n\ge1}
\lambda_n^m\,n^{-s}\Qp(n)^{-t}.
}
\label{eq:transport-moments}
\end{equation}
In particular,
\begin{equation}
\mathcal L^m Z_{\Qp}(s,0)
=(-1)^mD_m(s),
\qquad
D_m(s):=\sum_{n\ge1}\frac{\lambda_n^m}{n^s}.
\label{eq:Dm-def}
\end{equation}
\end{theorem}

\begin{proof}
In the coordinates \(w=s+t\), write
\[
Z_{\Qp}(s,t)=2^t\Zn(w,t).
\]
A direct calculation gives
\[
\mathcal L\bigl(2^tf(w,t)\bigr)=2^t\partial_tf(w,t).
\]
Since \(\Zn(w,t)=\sum_n e^{-t\lambda_n}n^{-w}\), repeated differentiation yields \eqref{eq:transport-moments}.  Taking \(f(w,t)=\zeta(w)\) gives \eqref{eq:baseline-transport}.
\end{proof}

Thus \(\mathcal L\) annihilates the macroscopic model exactly, and every application of \(\mathcal L\) extracts one additional logarithmic slope defect.

\begin{proposition}[Boundary hierarchy of defect moments]
\label{prop:moment-hierarchy}
As \(\sigma\downarrow1\),
\begin{align}
D_1(\sigma)&=O\!\left((\sigma-1)^{-1/2}\right),
\label{eq:D1-bound}\\
D_2(\sigma)&=O\!\left(\log\frac1{\sigma-1}\right),
\label{eq:D2-bound}\\
D_m(\sigma)&=O_m(1),\qquad m\ge3.
\label{eq:Dm-bound}
\end{align}
Moreover, for \(t\) in a compact set,
\begin{equation}
\Hz(w,t)
=-tD_1(w)+\frac{t^2}{2}D_2(w)+R_3(w,t),
\label{eq:second-order-H}
\end{equation}
where the defining series for \(R_3(w,t)\) converges absolutely even on the boundary line \(\RePart w=1\).
\end{proposition}

\begin{proof}
On the dyadic block \(I_k\), \eqref{eq:lambda-bound} gives \(|\lambda_n|^m\ll k^{-m/2}\).  At \(\sigma=1+\eps\), the block contribution to \(D_m\) is therefore \(\ll2^{-k\eps}k^{-m/2}\).  This yields \eqref{eq:D1-bound}--\eqref{eq:Dm-bound}.  Taylor's theorem gives
\[
e^{-t\lambda}-1
=-t\lambda+\frac{t^2\lambda^2}{2}+O_T(|\lambda|^3),
\]
and the \(m=3\) boundary series is absolutely convergent.
\end{proof}

The proposition identifies a useful hierarchy: at the level of the proved pointwise asymptotic, only the first two logarithmic moments can produce divergent boundary terms.  All higher Taylor modes are already summable on \(\RePart w=1\).

\section{A zeta-weighted probabilistic interpretation}

For real \(\sigma>1\), define a probability measure on \(\N\) by
\begin{equation}
\mathbb P_\sigma(N=n)=\frac{n^{-\sigma}}{\zeta(\sigma)}.
\label{eq:zeta-measure}
\end{equation}
Then \eqref{eq:normalized-zeta-decomp} becomes
\begin{equation}
{
\frac{2^{-t}Z_{\Qp}(\sigma-t,t)}{\zeta(\sigma)}
=\mathbb E_\sigma\!\left[e^{-t\lambda_N}\right].
}
\label{eq:laplace-expectation}
\end{equation}
Thus the two-variable zeta function is a Laplace transform of the logarithmic orbit defect under the classical zeta distribution.

\begin{corollary}[Concentration at the critical line]
\label{cor:zeta-concentration}
Uniformly for \(t\) in compact subsets of \(\C\),
\begin{equation}
\frac{2^{-t}Z_{\Qp}(\sigma-t,t)}{\zeta(\sigma)}
=1+O_T\!\left(\sqrt{\sigma-1}\right).
\label{eq:zeta-ratio}
\end{equation}
Furthermore,
\begin{align}
\mathbb E_\sigma|\lambda_N|
&=O\!\left(\sqrt{\sigma-1}\right),\\
\mathbb E_\sigma|\lambda_N|^2
&=O\!\left((\sigma-1)\log\frac1{\sigma-1}\right),\\
\mathbb E_\sigma|\lambda_N|^m
&=O_m(\sigma-1),\qquad m\ge3.
\end{align}
\end{corollary}

\begin{proof}
By \eqref{eq:normalized-zeta-decomp},
\[
 \frac{2^{-t}Z_{\Qp}(\sigma-t,t)}{\zeta(\sigma)}-1
 =\frac{\Hz(\sigma,t)}{\zeta(\sigma)}.
\]
Thus \eqref{eq:H-sqrt-bound} and
\(\zeta(\sigma)\asymp(\sigma-1)^{-1}\) prove
\eqref{eq:zeta-ratio}, uniformly for \(t\) in compact sets.

For the absolute moments, repeat the dyadic block estimate with absolute
values.  Put \(\eps:=\sigma-1\) and \(I_k=[2^k,2^{k+1})\).  By
\eqref{eq:lambda-bound}, for every \(m\ge1\),
\[
 \sum_{n\in I_k}\frac{|\lambda_n|^m}{n^\sigma}
 \ll_m 2^{-k\eps}(k+1)^{-m/2}.
\]
Consequently,
\[
 \sum_{n\ge1}\frac{|\lambda_n|}{n^\sigma}
 \ll\eps^{-1/2},
 \qquad
 \sum_{n\ge1}\frac{|\lambda_n|^2}{n^\sigma}
 \ll\log\frac1\eps,
 \qquad
 \sum_{n\ge1}\frac{|\lambda_n|^m}{n^\sigma}
 \ll_m1\quad(m\ge3).
\]
Using \eqref{eq:zeta-measure} and
\(\zeta(\sigma)\asymp\eps^{-1}\), division by \(\zeta(\sigma)\)
gives all three asserted absolute-moment estimates.
\end{proof}

In this sense the zeta measure concentrates the normalized slope \(2\Qp(n)/n\) at one as the critical line is approached from the right.

\section{Regrouping by values: frequencies and occurrence positions}

Every value of \(\Qp\) is odd, and every positive odd integer occurs with finite frequency \cite{Cloitre2026,Mantovanelli2026Frequency}.  Put
\begin{equation}
\mathcal N_m:=\{n\ge1:\Qp(n)=2m-1\},
\qquad
F(m):=|\mathcal N_m|.
\label{eq:level-sets}
\end{equation}
For each fixed \(m\), define the finite Dirichlet polynomial
\begin{equation}
A_m(s):=\sum_{n\in\mathcal N_m}n^{-s}.
\label{eq:Am}
\end{equation}

\begin{proposition}[Exact level-set decomposition]
\label{prop:level-decomp}
For \(\RePart(s+t)>1\),
\begin{equation}
{
Z_{\Qp}(s,t)
=\sum_{m\ge1}(2m-1)^{-t}A_m(s).
}
\label{eq:level-decomp}
\end{equation}
The two coordinate slices satisfy
\begin{equation}
A_m(0)=F(m),
\qquad
Z_{\Qp}(0,t)=\sum_{m\ge1}\frac{F(m)}{(2m-1)^t}
\qquad(\RePart t>1),
\label{eq:orbit-slice}
\end{equation}
and
\begin{equation}
\sum_{m\ge1}A_m(s)=\zeta(s)
\qquad(\RePart s>1).
\label{eq:index-slice}
\end{equation}
\end{proposition}

This formulation explains what the second variable adds.  The orbit zeta function at \(s=0\) remembers only the multiplicities \(F(m)\); the full \(s\)-dependence records where the occurrences of each value are located.

The proved asymptotic \eqref{eq:cloitre-asymptotic} allows a quantitative comparison with the ordinary frequency Dirichlet series.  If \(n\in\mathcal N_m\), then
\begin{equation}
\frac{n}{4m}=1+O\!\left((\log m)^{-1/2}\right)
\label{eq:occurrence-location}
\end{equation}
for \(m\to\infty\), uniformly over all occurrences of the value \(2m-1\).

Define the ordered frequency correction
\begin{equation}
C_F(w):=\sum_{m\ge1}\frac{F(m)-4}{m^w},
\qquad \RePart w>1.
\label{eq:CF}
\end{equation}
The exact dyadic frequency law implies the block-mass identity
\begin{equation}
\sum_{2^k\le m<2^{k+1}}F(m)=4\cdot2^k-1.
\label{eq:block-frequency-mass}
\end{equation}

\begin{theorem}[Frequency-position bridge]
\label{thm:frequency-position-bridge}
Let \(w=s+t\).  In \(\RePart w>1\),
\begin{equation}
{
Z_{\Qp}(s,t)
=2^{-2s-t}\bigl(4\zeta(w)+C_F(w)\bigr)
+E_{\mathrm{pos}}(s,t),
}
\label{eq:frequency-position-bridge}
\end{equation}
where
\begin{equation}
E_{\mathrm{pos}}(s,t)
=2^{-2s-t}
\sum_{m\ge1}\frac{1}{m^w}
\sum_{n\in\mathcal N_m}
\left[
\left(\frac{n}{4m}\right)^{-s}
\left(1-\frac1{2m}\right)^{-t}-1
\right].
\label{eq:Epos}
\end{equation}
For fixed \(t\) in a compact set and \(s=\sigma-t\),
\begin{equation}
E_{\mathrm{pos}}(\sigma-t,t)
=O_T\!\left((\sigma-1)^{-1/2}\right)
\qquad(\sigma\downarrow1).
\label{eq:Epos-bound}
\end{equation}
Consequently,
\begin{equation}
{
H_{\Qp}(s,t)
=2^t\bigl(4^{1-w}-1\bigr)\zeta(w)
+2^{-2s-t}C_F(w)
+E_{\mathrm{pos}}(s,t).
}
\label{eq:H-frequency-position}
\end{equation}
The first term on the right has a removable singularity at \(w=1\).
\end{theorem}

\begin{proof}
For \(n\in\mathcal N_m\), factor
\[
n^{-s}(2m-1)^{-t}
=2^{-2s-t}m^{-w}
\left(\frac{n}{4m}\right)^{-s}
\left(1-\frac1{2m}\right)^{-t}.
\]
Summing first over the finite set \(\mathcal N_m\) gives \eqref{eq:frequency-position-bridge} and \eqref{eq:Epos}.

By \eqref{eq:occurrence-location}, the bracket in \eqref{eq:Epos} is
\(O_T((\log m)^{-1/2})\).  On a dyadic value block, the total multiplicity is \(O(2^k)\) by \eqref{eq:block-frequency-mass}; hence the block contribution at \(w=1+\eps\) is
\(O_T(2^{-k\eps}k^{-1/2})\).  Summing proves \eqref{eq:Epos-bound}.  Finally, substitute
\(\sum_mF(m)m^{-w}=4\zeta(w)+C_F(w)\) and subtract \(2^t\zeta(w)\).  Since
\(4\cdot2^{-2s-t}=2^t4^{1-w}\), \eqref{eq:H-frequency-position} follows.
\end{proof}

Equation \eqref{eq:H-frequency-position} separates two genuinely different sources of fine structure: the ordered frequency discrepancy \(C_F\) and the displacement of the actual occurrences from the macroscopic location \(n=4m\).

\section{Exact dyadic renormalization in the index variable}

The two-variable zeta function also admits an exact scale splitting.  Define
\begin{equation}
r_n^+:=\frac{\Qp(2n)}{2\Qp(n)},
\qquad
r_n^-:=\frac{\Qp(2n-1)}{2\Qp(n)}.
\label{eq:rplusminus}
\end{equation}
By \eqref{eq:cloitre-asymptotic},
\begin{equation}
r_n^\pm=1+O\!\left((\log n)^{-1/2}\right).
\label{eq:rplusminus-asymptotic}
\end{equation}
The even ratio contains the familiar dyadic defect explicitly:
\[
r_n^+-1
=\frac{\Qp(2n)-2\Qp(n)}{2\Qp(n)}.
\]

For \(w=s+t\), define the scale-defect transform
\begin{align}
\mathscr R(s,t)
:=\sum_{n\ge1}n^{-s}\Qp(n)^{-t}
\Bigg[& (r_n^+)^{-t}
+\left(1-\frac1{2n}\right)^{-s}(r_n^-)^{-t}
-2\Bigg].
\label{eq:scale-defect-transform}
\end{align}

\begin{theorem}[Exact dyadic renormalization equation]
\label{thm:dyadic-renorm}
For \(\RePart(s+t)>1\),
\begin{equation}
\boxed{
\bigl(1-2^{1-s-t}\bigr)Z_{\Qp}(s,t)
=2^{-s-t}\mathscr R(s,t).
}
\label{eq:dyadic-renorm}
\end{equation}
Equivalently, with \(w=s+t\),
\begin{equation}
\bigl(1-2^{1-w}\bigr)Z_{\Qp}(w-t,t)
=2^{-w}\mathscr R(w-t,t).
\label{eq:dyadic-renorm-w}
\end{equation}
Furthermore, for each fixed \(t\in\C\),
\begin{equation}
{
\lim_{\sigma\downarrow1}
\mathscr R(\sigma-t,t)
=2^{t+1}\log2.
}
\label{eq:R-critical-limit}
\end{equation}
\end{theorem}

\begin{proof}
Split \(Z_{\Qp}\) into even and odd indices.  The even terms are
\[
(2n)^{-s}\Qp(2n)^{-t}
=2^{-w}n^{-s}\Qp(n)^{-t}(r_n^+)^{-t},
\]
and the odd terms are
\[
(2n-1)^{-s}\Qp(2n-1)^{-t}
=2^{-w}n^{-s}\Qp(n)^{-t}
\left(1-\frac1{2n}\right)^{-s}(r_n^-)^{-t}.
\]
Adding and subtracting twice the original summand gives \eqref{eq:dyadic-renorm}.

For \eqref{eq:R-critical-limit}, rearrange \eqref{eq:dyadic-renorm-w} and use \cref{thm:abelian-principal}:
\[
\mathscr R(\sigma-t,t)
=2^\sigma\bigl(1-2^{1-\sigma}\bigr)Z_{\Qp}(\sigma-t,t).
\]
Since \(1-2^{1-\sigma}\sim(\sigma-1)\log2\), the limit is \(2^{t+1}\log2\).
\end{proof}

The denominator in \eqref{eq:dyadic-renorm-w} has the scale lattice
\begin{equation}
w_\ell=1+\frac{2\pi i\ell}{\log2},
\qquad \ell\in\mathbb Z.
\label{eq:scale-lattice}
\end{equation}
This is the exact analytic footprint of doubling the index.

\begin{corollary}[Continuation criterion from the scale defect]
\label{cor:scale-continuation}
Fix \(t_0\in\C\) and a lattice point \(w_\ell\) from \eqref{eq:scale-lattice}.  If the function
\[
(w,t)\longmapsto\mathscr R(w-t,t)
\]
admits a holomorphic continuation to a neighborhood of \((w_\ell,t_0)\), then
\(Z_{\Qp}(w-t,t)\) admits a meromorphic continuation there, with at most a simple pole along \(w=w_\ell\).  Its residue in the \(w\)-variable is
\begin{equation}
\operatorname*{Res}_{w=w_\ell}Z_{\Qp}(w-t,t)
=\frac{2^{-w_\ell}\mathscr R(w_\ell-t,t)}{\log2}.
\label{eq:scale-residue}
\end{equation}
At \(w_0=1\), the right-boundary value \eqref{eq:R-critical-limit} gives the expected residue \(2^t\).
\end{corollary}

This criterion does not prove continuation by itself.  It does, however, reduce the problem to a transform of the \emph{failure of exact dyadic covariance}; the universal scale denominator has been completely separated.

\section{The exceptional slice \texorpdfstring{$t=-1$}{t=-1}: an unconditional continuation}

The slice \(t=-1\) is exceptional because the normalized correction becomes the ordinary Dirichlet series of the centered orbit error.  The exact arch structure then reveals a cancellation that is invisible in the pointwise estimate \eqref{eq:cloitre-asymptotic}.

At \(t=-1\),
\begin{equation}
\Zn(w,-1)
=2Z_{\Qp}(w+1,-1)
=\sum_{n\ge1}\frac{2\Qp(n)}{n^{w+1}},
\label{eq:tminus1-Z}
\end{equation}
and therefore
\begin{equation}
{
\Hz(w,-1)
=\sum_{n\ge1}\frac{2\Qp(n)-n}{n^{w+1}}.
}
\label{eq:tminus1-H}
\end{equation}
Thus this slice is the ordinary Dirichlet series of the centered orbit error
\begin{equation}
E(n):=2\Qp(n)-n.
\label{eq:centered-E}
\end{equation}

\subsection{An exact center-of-mass clock}

Recall the slow odd/even branches
\begin{equation}
U(m):=\frac{\Qp(2m-1)+1}{2},
\qquad
V(m):=\frac{\Qp(2m)+1}{2}.
\label{eq:UV-branches}
\end{equation}
Here and below, \(\Delta f(m):=f(m+1)-f(m)\) denotes the forward
difference.  The exact arch construction was proved by Clo\^{\i}tre
\cite{Cloitre2026} and recalled in the author's frequency-law paper
\cite[Eqs.~(11)--(19)]{Mantovanelli2026Frequency}.  In the present notation
it uses
\begin{equation}
a_r=\frac{2\cdot4^{r+1}+1}{3},
\qquad
\nu_r=2a_r-r-2,
\qquad
v_r=4a_r-r-2.
\label{eq:arch-coordinates}
\end{equation}
We denote the left endpoint by \(\nu_r\) (rather than the more common
\(u_r\)) to avoid confusion with the complex zeta variable.
The following proposition records the structural input in the notation used
here.  In particular, no asymptotic or numerical observation enters these
identities.

\begin{proposition}[Exact arch data]
\label{prop:exact-arch-data}
The positive arch occupies \([\nu_r,v_r]\), and the following negative arch
ends at \(\nu_{r+1}\).  If \(P_r\) and \(N_r\) denote the corresponding
binary increment words, then
\begin{align}
\Delta U(\nu_r+t)&=P_r[t],
&\Delta V(\nu_r+t)&=1-P_r[t]
&& (0\le t<2a_r),
\label{eq:positive-increments}\\
\Delta V(v_r+t)&=N_r[t]
&&& (0\le t<4a_r-3),
\label{eq:negative-V}\\
\Delta U(v_r)&=1,
\label{eq:negative-U-first}\\
\Delta U(v_r+t)&=1-N_r[t-1]
&&& (1\le t<4a_r-3).
\label{eq:negative-U}
\end{align}
At the arch endpoints,
\begin{equation}
U(\nu_r)=V(\nu_r)=a_r,
\qquad
U(v_r)=V(v_r)=2a_r,
\qquad
U(\nu_{r+1})=V(\nu_{r+1})=a_{r+1}.
\label{eq:arch-endpoints}
\end{equation}
\end{proposition}

\begin{proof}
This is Clo\^{\i}tre's exact positive/negative arch decomposition
\cite{Cloitre2026}, recalled in
\cite[Eqs.~(11)--(19)]{Mantovanelli2026Frequency}, written for the slow
branches \eqref{eq:UV-branches} and with the left endpoint denoted by
\(\nu_r\).
The cited construction gives the four increment identities together with
the three endpoint values; the lengths in the displayed ranges follow at
once from \eqref{eq:arch-coordinates}.  We have stated the complete input
needed below so that every subsequent use can be traced to a displayed
identity.
\end{proof}

Define the center-of-mass clock
\begin{equation}
C(m):=U(m)+V(m)-m.
\label{eq:C-clock}
\end{equation}

\begin{theorem}[Exact center-of-mass clock]
\label{thm:center-clock}
For every \(r\ge0\),
\begin{equation}
C(m)=r+2
\qquad(\nu_r\le m\le v_r).
\label{eq:C-positive}
\end{equation}
Put \(L_r:=4a_r-3=\nu_{r+1}-v_r\).  Then
\begin{equation}
C(v_r)=r+2,
\qquad
C(v_r+t)=r+2+N_r[t-1]
\quad(1\le t\le L_r).
\label{eq:C-negative}
\end{equation}
In particular, throughout the complete \(r\)-th positive/negative arch pair,
\begin{equation}
C(m)\in\{r+2,r+3\}.
\label{eq:C-two-values}
\end{equation}
\end{theorem}

\begin{proof}
From \eqref{eq:positive-increments},
\[
\Delta C(\nu_r+t)
=\Delta U(\nu_r+t)+\Delta V(\nu_r+t)-1=0.
\]
Since \eqref{eq:arch-endpoints} and \eqref{eq:arch-coordinates} give
\[
C(\nu_r)=2a_r-\nu_r=r+2,
\]
this proves \eqref{eq:C-positive}.

On the negative arch, \eqref{eq:negative-V}--\eqref{eq:negative-U} give
\[
\Delta C(v_r)=N_r[0]
\]
and, for \(1\le t<L_r\),
\[
\Delta C(v_r+t)=N_r[t]-N_r[t-1].
\]
Telescoping from \(C(v_r)=r+2\) yields \eqref{eq:C-negative}.  The endpoint formula
\(C(\nu_{r+1})=r+3\) is consistent with the final bit \(N_r[L_r-1]=1\).
\end{proof}

The centered error has a particularly simple pair sum.

\begin{corollary}[Paired centered error]
\label{cor:paired-error}
For every \(m\ge1\),
\begin{equation}
{
E(2m-1)+E(2m)=4C(m)-3.
}
\label{eq:paired-error}
\end{equation}
Consequently,
\begin{equation}
E(2m-1)+E(2m)=2\log_2m+O(1)
\qquad(m\to\infty).
\label{eq:paired-error-asymptotic}
\end{equation}
\end{corollary}

\begin{proof}
Using \eqref{eq:UV-branches},
\[
E(2m-1)=4U(m)-2m-1,
\qquad
E(2m)=4V(m)-2m-2,
\]
which gives \eqref{eq:paired-error}.  If \(\nu_r\le m<\nu_{r+1}\), then \cref{thm:center-clock} gives \(C(m)=r+O(1)\), while \eqref{eq:arch-coordinates} gives \(m\asymp4^r\).  Hence
\(r=\tfrac12\log_2m+O(1)\), proving \eqref{eq:paired-error-asymptotic}.
\end{proof}

The clock is sufficiently explicit to give closed cycle sums.  Since
\[
L_r=4a_r-3,
\qquad
\sum_{t=0}^{L_r-1}N_r[t]
=V(\nu_{r+1})-V(v_r)
=2a_r-1,
\]
and the last bit is one, the number of ones among \(N_r[0],\ldots,N_r[L_r-2]\) is \(2a_r-2\).  Therefore
\begin{equation}
\sum_{m=\nu_r}^{\nu_{r+1}-1}C(m)
=(6a_r-3)(r+2)+2a_r-2.
\label{eq:C-cycle-sum}
\end{equation}
Combining this with \eqref{eq:paired-error} gives the exact pair-error mass
\begin{equation}
{
\sum_{m=\nu_r}^{\nu_{r+1}-1}
\bigl(E(2m-1)+E(2m)\bigr)
=(6a_r-3)(4r+5)+8a_r-8.
}
\label{eq:E-cycle-sum}
\end{equation}
Since \(6a_r-3=16\cdot4^r-1\), this is an explicit affine-polynomial function of \(r\) times \(4^r\).

\subsection{A summatory theorem}

Put
\begin{equation}
A(X):=\sum_{n\le X}E(n).
\label{eq:AX}
\end{equation}

\begin{theorem}[Summatory centered-error law]
\label{thm:centered-error-asymptotic}
As \(X\to\infty\),
\begin{equation}
{
A(X)=X\log_2X+O(X).
}
\label{eq:A-asymptotic}
\end{equation}
Moreover, at the exact arch endpoints
\begin{equation}
X_R:=2\nu_R-2,
\label{eq:XR}
\end{equation}
one has, for every \(R\ge1\),
\begin{equation}
{
A(X_R)
=\frac{16}{3}(4R+1)4^R
-2R^2-\frac{25}{3}R-\frac73.
}
\label{eq:A-exact-endpoint}
\end{equation}
Consequently,
\begin{equation}
\frac{A(X_R)-X_R\log_2X_R}{X_R}
\longrightarrow
\frac12-\log_2\frac{32}{3}.
\label{eq:A-second-term-subsequence}
\end{equation}
\end{theorem}

\begin{proof}
By \eqref{eq:paired-error-asymptotic},
\[
A(2M)
=\sum_{m\le M}\bigl(E(2m-1)+E(2m)\bigr)
=2\sum_{m\le M}\log_2m+O(M).
\]
Stirling's formula gives
\[
2\sum_{m\le M}\log_2m
=2M\log_2M+O(M),
\]
and hence
\[
A(2M)=2M\log_2(2M)+O(M).
\]
For an odd endpoint, the unmatched term \(E(2M+1)\) is \(O(M)\) by \eqref{eq:cloitre-asymptotic}; this proves \eqref{eq:A-asymptotic}.

For the exact formula, note that the first three pair sums equal one, so their total is three.  The cycles \(r=0,\ldots,R-1\) cover all \(m\) from \(4=\nu_0\) through \(\nu_R-1\).  Summing \eqref{eq:E-cycle-sum} and using
\[
\sum_{r=0}^{R-1}4^r=\frac{4^R-1}{3},
\qquad
\sum_{r=0}^{R-1}r4^r
=\frac{4+(3R-4)4^R}{9},
\]
gives \eqref{eq:A-exact-endpoint}.  Finally, \(X_R=(32/3)4^R+O(R)\), and \eqref{eq:A-second-term-subsequence} follows by direct division.
\end{proof}

\Cref{fig:centered-error} illustrates the convergence established by the
theorem.

\begin{figure}[htbp]
\centering
\includegraphics[width=0.88\textwidth]{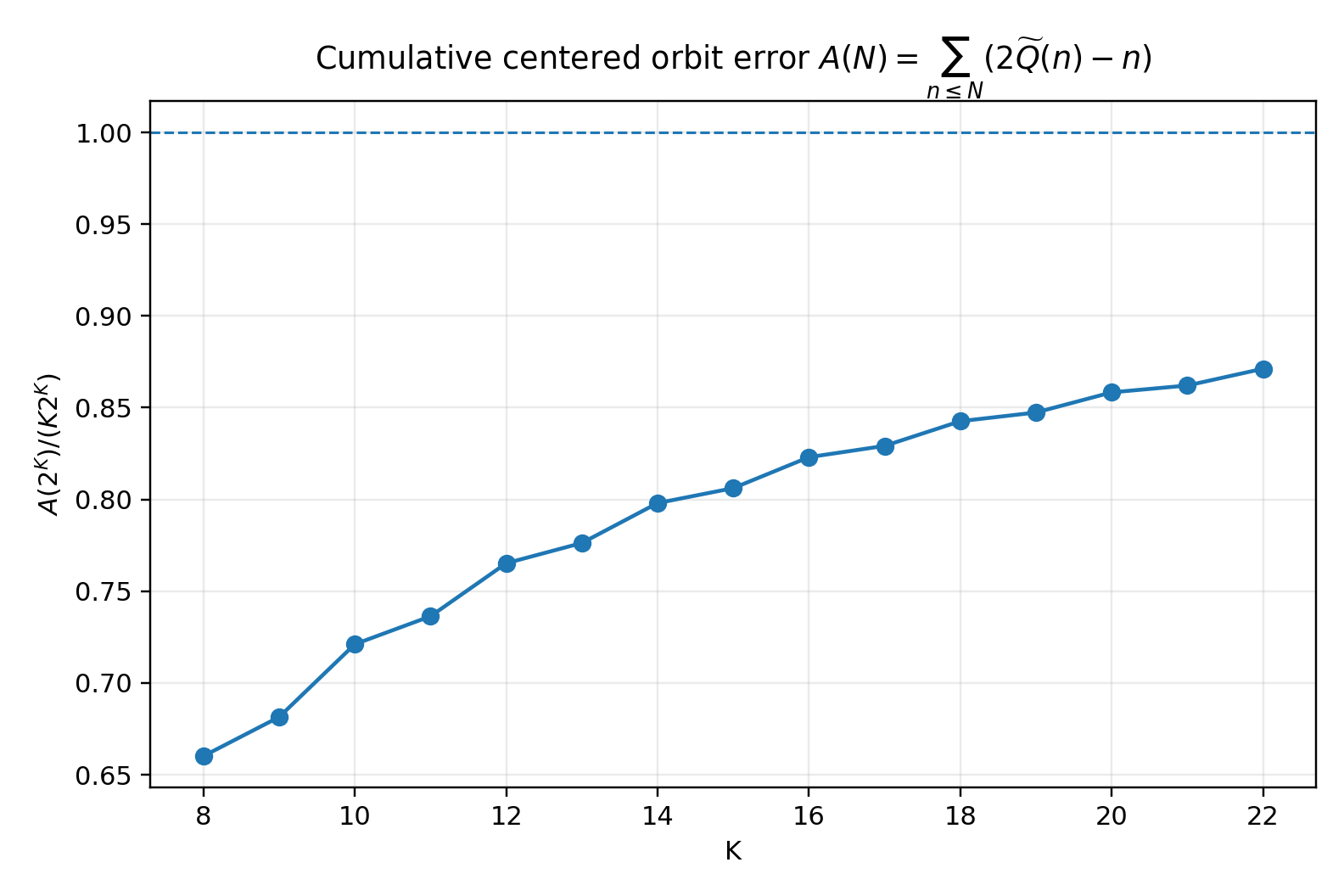}
\caption{The normalized cumulative centered error at dyadic endpoints.  The
dashed line marks the limiting value \(1\) in \eqref{eq:A-asymptotic}.}
\label{fig:centered-error}
\end{figure}

\subsection{Analytic continuation of the linear slice}

The summatory theorem immediately crosses the original convergence boundary.

\begin{theorem}[Unconditional continuation of the \(t=-1\) slice]
\label{thm:tminus1-continuation}
The series \eqref{eq:tminus1-H} converges locally uniformly and defines a holomorphic function in
\begin{equation}
{\RePart w>0.}
\label{eq:tminus1-strip}
\end{equation}
Consequently,
\begin{equation}
{
Z_{\Qp}(s,-1)
=\frac12\zeta(s-1)+\frac12\Hz(s-1,-1)
}
\label{eq:tminus1-decomp}
\end{equation}
admits a meromorphic continuation from its original half-plane \(\RePart s>2\) to the larger half-plane
\begin{equation}
{\RePart s>1.}
\label{eq:tminus1-s-strip}
\end{equation}
Within this half-plane, the only forced singularity from the zeta term is the simple pole at \(s=2\).
\end{theorem}

\begin{proof}
Partial summation gives, initially for \(\RePart w>1\),
\[
\sum_{n\le X}\frac{E(n)}{n^{w+1}}
=A(X)X^{-w-1}
+(w+1)\int_1^X A(x)x^{-w-2}\,dx.
\]
By \cref{thm:centered-error-asymptotic}, \(A(x)=O(x\log x)\).  Hence the boundary term tends to zero and the integral converges locally uniformly for every \(\RePart w>0\), proving the continuation.
\end{proof}

The same asymptotic also identifies the next boundary scale.  Define
\begin{equation}
r(n):=E(n)-\frac{\log n+1}{\log2}.
\label{eq:r-residual}
\end{equation}
By Stirling's formula and \eqref{eq:A-asymptotic},
\begin{equation}
\sum_{n\le X}r(n)=O(X).
\label{eq:r-partial-sums}
\end{equation}
Therefore, for \(\RePart w>0\),
\begin{equation}
{
\Hz(w,-1)
=\frac{-\zeta'(w+1)+\zeta(w+1)}{\log2}
+R_*(w),
}
\label{eq:tminus1-zeta-derivative}
\end{equation}
where
\begin{equation}
R_*(w):=\sum_{n\ge1}\frac{r(n)}{n^{w+1}}
\label{eq:Rstar}
\end{equation}
converges locally uniformly in \(\RePart w>0\).  Using the Laurent expansion
of \(\zeta\) at \(1\) \cite{Titchmarsh1986}, we obtain, on the positive real
axis,
\begin{equation}
{
\Hz(w,-1)
=\frac{1}{\log2}\frac1{w^2}+O\!\left(\frac1w\right)
\qquad(w\downarrow0).
}
\label{eq:tminus1-double-boundary}
\end{equation}
Thus the first nontrivial continued slice already exhibits a second layer of zeta structure: after the pole at \(s=2\), the approach to the next boundary \(s=1\) is governed by a zeta derivative.

Equation \eqref{eq:A-second-term-subsequence} shows that a nontrivial order-$X$ second term remains after the smooth term $X\log_2X$.  The next subsections identify this discrete-scale term explicitly, compute its Fourier spectrum, and determine the boundary resonance that it forces on $\RePart w=0$.

\subsection{The explicit log-periodic second term}

The order-$X$ remainder mentioned above can in fact be identified explicitly.
The key point for the even partial sums $A(2M)$ is that the positive arch
has a constant center-of-mass clock, while on the negative arch the
remaining fluctuation is a prefix imbalance of the binary word $N_r$.

For $0\le t\le L_r$ define
\begin{equation}
\mathscr D_r(t)
:=2\sum_{j=0}^{t-1}N_r[j]-t,
\qquad \mathscr D_r(0):=0.
\label{eq:negative-prefix-imbalance}
\end{equation}
Thus $\mathscr D_r(t)$ is the excess of ones over zeros in the first $t$
bits of the negative-arch word.

\begin{proposition}[Exact local second-order skeleton]
\label{prop:exact-local-skeleton}
For every $r\ge0$ the following identities hold.
If $\nu_r\le M\le v_r$, then
\begin{equation}
{
A(2M)
=(4r+5)M-\frac{64}{3}4^r
+\frac{2}{3}(3r^2+9r+14).
}
\label{eq:positive-local-skeleton}
\end{equation}
If $M=v_r+t$ with $0\le t\le L_r-1$, then
\begin{equation}
{
A(2M)
=(4r+7)M-\frac{128}{3}4^r
+\frac{2}{3}(3r^2+12r+16)
+2\mathscr D_r(t).
}
\label{eq:negative-local-skeleton}
\end{equation}
Moreover,
\begin{equation}
{
\mathscr D_r(t)
=\Qp\!\left(2(v_r+t)\right)-\Qp(2v_r)-t.
}
\label{eq:prefix-imbalance-Q}
\end{equation}
Consequently, as $r\to\infty$, uniformly for $0\le t\le L_r$,
\begin{equation}
\mathscr D_r(t)
=O\!\left(\frac{4^r}{\sqrt r}\right).
\label{eq:prefix-imbalance-bound}
\end{equation}
\end{proposition}

\begin{proof}
On the positive arch, \cref{thm:center-clock,cor:paired-error} gives
\[
E(2m-1)+E(2m)=4r+5.
\]
Starting from the exact value \eqref{eq:A-exact-endpoint} at
$2\nu_r-2$ and summing from $m=\nu_r$ to $M$ gives
\eqref{eq:positive-local-skeleton} after inserting
$\nu_r=(16\cdot4^r-3r-4)/3$.

On the negative arch,
\[
E(2(v_r+j)-1)+E(2(v_r+j))
=4r+5+4N_r[j-1]
\qquad(j\ge1).
\]
Summing this identity from $v_r$ and using the positive formula at $v_r$
gives \eqref{eq:negative-local-skeleton}.  Finally,
\eqref{eq:negative-V} implies
\[
\sum_{j=0}^{t-1}N_r[j]=V(v_r+t)-V(v_r).
\]
Using $2V(m)=\Qp(2m)+1$ proves \eqref{eq:prefix-imbalance-Q}.  Since
$v_r+t\asymp4^r$ throughout the negative arch, the asymptotic
\eqref{eq:cloitre-asymptotic} gives \eqref{eq:prefix-imbalance-bound}
uniformly in $t$.
\end{proof}

Define a continuous $1$-periodic function by prescribing it on $[0,1]$:
\begin{equation}
{
\Omega(u)
:=\frac32-\log_2\frac{16}{3}-u-2^{1-u},
\qquad 0\le u\le1,
\qquad \Omega(u+1)=\Omega(u).
}
\label{eq:Omega-def}
\end{equation}
The endpoint values agree, so the periodic extension is continuous (and
Lipschitz), although its derivative has a jump at the integers.

\begin{theorem}[Explicit log-periodic second term]
\label{thm:log-periodic-second-term}
As $X\to\infty$,
\begin{equation}
\boxed{
A(X)
=X\log_2X
+X\,\Omega\!\left(\log_2\frac{3X}{32}\right)
+O\!\left(\frac{X}{\sqrt{\log X}}\right).
}
\label{eq:log-periodic-A}
\end{equation}
The error estimate is uniform in the phase of $X$.
\end{theorem}

\begin{proof}
First let $X=2M$ and choose $r$ so that
$\nu_r\le M<\nu_{r+1}$.  Put $c=16/3$.
On the ideal positive scale $c4^r\le M\le2c4^r$, direct substitution into
\eqref{eq:Omega-def} gives
\begin{equation}
2M\log_2(2M)
+2M\Omega\!\left(\log_2\frac{3M}{16}\right)
=(4r+5)M-\frac{64}{3}4^r.
\label{eq:Omega-positive-linear}
\end{equation}
On the ideal negative scale $2c4^r\le M\le4c4^r$, periodicity gives
\begin{equation}
2M\log_2(2M)
+2M\Omega\!\left(\log_2\frac{3M}{16}\right)
=(4r+7)M-\frac{128}{3}4^r.
\label{eq:Omega-negative-linear}
\end{equation}
The exact arch endpoints differ from the ideal points
$c4^r,2c4^r,4c4^r$ by only $O(r)$.  Since $\Omega$ is Lipschitz, replacing
the ideal partition by the exact arch partition changes the right-hand sides
of \eqref{eq:Omega-positive-linear}--\eqref{eq:Omega-negative-linear} by
$O(r)$.  The exact identities of \cref{prop:exact-local-skeleton} therefore
give an error $O(r^2)$ on the positive arch and
\[
2\mathscr D_r(t)+O(r^2)
=O\!\left(\frac{4^r}{\sqrt r}\right)
\]
on the negative arch.  Since $X\asymp4^r$ and $r\asymp\log X$, this proves
\eqref{eq:log-periodic-A} for even $X$.

For odd $X=2M+1$, the difference from the preceding even partial sum is
$E(2M+1)=O(M/\sqrt{\log M})$ by \eqref{eq:cloitre-asymptotic}.  The model
function on the right-hand side of \eqref{eq:log-periodic-A} changes by only
$O(\log X)$ over a unit increment because $\Omega$ is Lipschitz.  This is
absorbed by the same error term.
\end{proof}

The theorem refines the first-order law in a directly testable way.
At powers of two the phase is fixed and the periodic value simplifies.

\begin{corollary}[Dyadic second term]
\label{cor:dyadic-second-term}
As $K\to\infty$,
\begin{equation}
{
A(2^K)
=K2^K-\frac{17}{6}2^K
+O\!\left(\frac{2^K}{\sqrt K}\right).
}
\label{eq:dyadic-A-second-term}
\end{equation}
\end{corollary}

\begin{proof}
The fractional phase is
$\{\log_2 3\}=\log_2(3/2)$.  Substitution into
\eqref{eq:Omega-def} gives
\[
\Omega(\log_2(3/2))
=\frac32-\log_2\frac{16}{3}-\log_2\frac32-\frac43
=-\frac{17}{6}.
\]
Now apply \cref{thm:log-periodic-second-term}.
\end{proof}

\begin{figure}[htbp]
\centering
\includegraphics[width=0.88\textwidth]{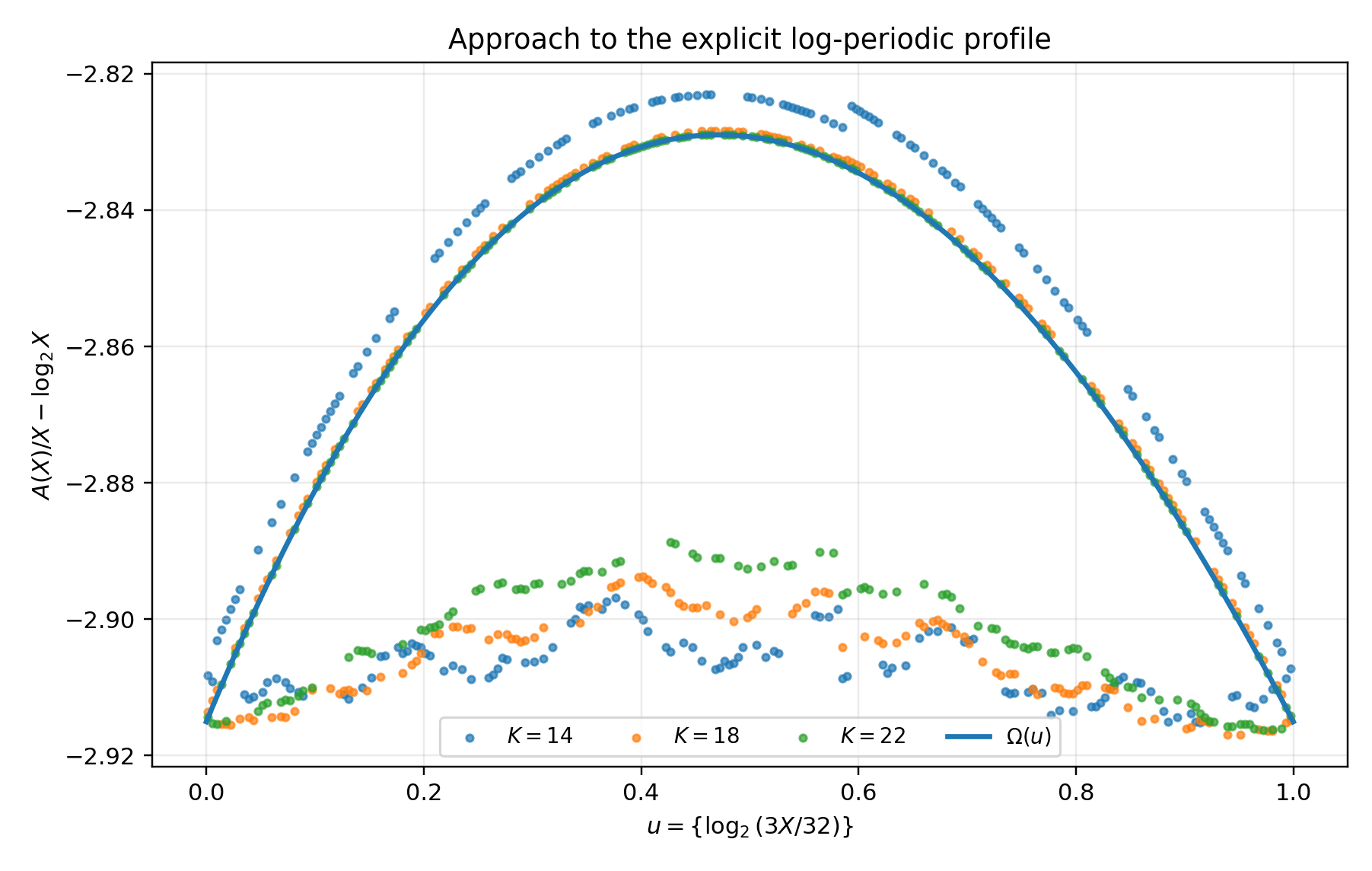}
\caption{Direct recurrence data compared with the explicit profile
$\Omega(u)$.  The order-$X$ oscillation is already captured by the proved
log-periodic term.  The second visible branch is smaller by the rigorously
identified square-root logarithmic scale and comes from the negative-arch prefix
imbalance $\mathscr D_r$.}
\label{fig:logperiodic-collapse}
\end{figure}

\subsection{Fourier spectrum and an Abelian pole skeleton}

The profile \eqref{eq:Omega-def} has an elementary Fourier transform.
Write
\begin{equation}
\Omega(u)=\sum_{m\in\mathbb Z}\widehat\Omega_m e^{2\pi imu}.
\label{eq:Omega-Fourier}
\end{equation}
Direct integration gives
\begin{equation}
{
\widehat\Omega_0
=\log_2\frac38-\frac1{\log2},
}
\label{eq:Omega-mean}
\end{equation}
and, for $m\ne0$,
\begin{equation}
{
\widehat\Omega_m
=\frac{\log2}
{(2\pi i m)(\log2+2\pi i m)}.
}
\label{eq:Omega-Fourier-coeff}
\end{equation}
Since \eqref{eq:Omega-Fourier-coeff} is $O(m^{-2})$, the Fourier series \eqref{eq:Omega-Fourier} converges absolutely and uniformly.
Here and below $\log2$ means the natural logarithm of $2$.

Put
\begin{equation}
\tau_m:=\frac{2\pi m}{\log2},
\qquad
d_m:=\widehat\Omega_m e^{2\pi i m\log_2 3}.
\label{eq:tau-d-def}
\end{equation}
Then \cref{thm:log-periodic-second-term} is equivalently
\begin{equation}
A(X)
=X\log_2X+d_0X
+\sum_{m\ne0}d_m X^{1+i\tau_m}
+O\!\left(\frac{X}{\sqrt{\log X}}\right),
\label{eq:A-Fourier-expansion}
\end{equation}
with an absolutely and uniformly convergent oscillatory sum.

The Fourier expansion identifies the full leading boundary resonance of the
$t=-1$ Dirichlet series.  Define
\begin{equation}
\mathcal M_{\Omega}(w)
:=\frac{w+1}{(\log2)w^2}
+(w+1)\sum_{m\in\mathbb Z}\frac{d_m}{w-i\tau_m}.
\label{eq:M-Omega-def}
\end{equation}
The series in \eqref{eq:M-Omega-def} converges normally away from the lattice
$i\tau_m$, so $\mathcal M_{\Omega}$ is meromorphic in the whole plane.

\begin{theorem}[Abelian boundary resonance]
\label{thm:Abelian-boundary-resonance}
Let $w=\sigma+i\tau$ with $0<\sigma\le1$ and $|\tau|\le T$.  Then
\begin{equation}
{
\Hz(w,-1)-\mathcal M_{\Omega}(w)
=O_T(\sigma^{-1/2})
\qquad(\sigma\downarrow0).
}
\label{eq:H-minus-M-bound}
\end{equation}
In particular,
\begin{equation}
{
\Hz(\sigma,-1)
=\frac{1}{(\log2)\sigma^2}
+\frac{\log_2(3/8)}{\sigma}
+O(\sigma^{-1/2}),
}
\label{eq:H-zero-boundary-refined}
\end{equation}
and, for every fixed $m\ne0$,
\begin{equation}
{
\Hz(\sigma+i\tau_m,-1)
=\frac{3^{i\tau_m}}{2\pi i m}\frac1\sigma
+O_m(\sigma^{-1/2}).
}
\label{eq:H-lattice-boundary}
\end{equation}
Thus every point $i\tau_m$ is a nonremovable boundary singularity of the
right-half-plane function $\Hz(w,-1)$.
\end{theorem}

\begin{proof}
Partial summation gives, for $\RePart w>0$,
\[
\Hz(w,-1)
=(w+1)\int_1^\infty A(x)x^{-w-2}\,dx.
\]
Insert \eqref{eq:A-Fourier-expansion}.  The first three model terms integrate
to \eqref{eq:M-Omega-def}.  The remainder is bounded in absolute value by a
constant times
\[
\int_3^\infty
\frac{x^{-\sigma-1}}{\sqrt{\log x}}\,dx
=O(\sigma^{-1/2}),
\]
uniformly for $|\tau|\le T$.  This proves \eqref{eq:H-minus-M-bound}.

At $w=0$, the double-pole term contributes
$1/((\log2)w^2)+1/((\log2)w)$, while the $m=0$ Fourier mode contributes
$d_0/w+O(1)$.  Equation \eqref{eq:Omega-mean} makes the simple coefficient
$1/\log2+d_0=\log_2(3/8)$, giving
\eqref{eq:H-zero-boundary-refined}.

For $m\ne0$, the residue of the corresponding term of
\eqref{eq:M-Omega-def} is
\[
(1+i\tau_m)d_m
=\frac{e^{2\pi i m\log_2 3}}{2\pi i m}
=\frac{3^{i\tau_m}}{2\pi i m},
\]
which gives \eqref{eq:H-lattice-boundary}.  Since the remainder is only
$O(\sigma^{-1/2})$, it cannot cancel the displayed $1/\sigma$ or
$1/\sigma^2$ divergence.
\end{proof}

There is a compact closed form for the same pole skeleton.  Define
\begin{equation}
\boxed{
\mathcal K(w)
:=\frac{3^w}{w(2^w-1)}-\frac{5}{2w}.
}
\label{eq:compact-pole-skeleton}
\end{equation}

\begin{proposition}[Closed meromorphic skeleton]
\label{prop:closed-pole-skeleton}
The difference
\begin{equation}
\mathcal M_{\Omega}(w)-\mathcal K(w)
\label{eq:M-minus-K}
\end{equation}
is entire.  Hence \eqref{eq:compact-pole-skeleton} contains exactly the
principal parts forced by the log-periodic second term.
\end{proposition}

\begin{proof}
Both functions are meromorphic with possible poles only at
$w=i\tau_m$.  If $m\ne0$, then
\[
\operatorname*{Res}_{w=i\tau_m}\mathcal K(w)
=\frac{3^{i\tau_m}}{i\tau_m\log2}
=\frac{3^{i\tau_m}}{2\pi i m},
\]
which agrees with \cref{thm:Abelian-boundary-resonance}.  At $w=0$,
expansion of \eqref{eq:compact-pole-skeleton} gives
\[
\mathcal K(w)
=\frac1{(\log2)w^2}
+\frac{\log_2(3/8)}{w}
+O(1),
\]
again matching \eqref{eq:M-Omega-def}.  All principal parts therefore cancel,
and the difference is entire.
\end{proof}

For the original two-variable function this yields a precise obstruction at
the next boundary.  Since
\[
Z_{\Qp}(s,-1)
=\frac12\zeta(s-1)+\frac12\Hz(s-1,-1),
\]
there can be no holomorphic continuation through any point
\begin{equation}
s=1+\frac{2\pi i m}{\log2}.
\label{eq:s-boundary-lattice}
\end{equation}
If a meromorphic continuation across such a point exists, its principal part
is forced by \eqref{eq:H-zero-boundary-refined}--\eqref{eq:H-lattice-boundary}.
In particular the point $s=1$ would carry a double pole, while the nonzero
lattice points would carry simple poles.

\subsection{The next layer: exact negative layers and a Gaussian skeleton}

The order-$X$ log-periodic term does not exhaust the arch structure.  Define
\begin{equation}
\mathcal R_A(X)
:=A(X)-X\log_2X
-X\Omega\!\left(\log_2\frac{3X}{32}\right).
\label{eq:RA-def}
\end{equation}
The proof of \cref{thm:log-periodic-second-term} and the exact local formulas
show more than the global bound:
\begin{equation}
\mathcal R_A(2M)=O(r^2)
\qquad(\nu_r\le M\le v_r),
\label{eq:RA-positive-small}
\end{equation}
whereas on the negative arch, for $M=v_r+t$,
\begin{equation}
{
\mathcal R_A(2M)=2\mathscr D_r(t)+O(r^2).
}
\label{eq:RA-negative-D}
\end{equation}
Thus on the even lattice the next fluctuation is reduced to the single
prefix-imbalance process \eqref{eq:negative-prefix-imbalance}.  From this
point through \cref{cor:flat-area-polylog} we study precisely this
negative-even arch channel.  These results do not constitute a
decomposition of the full Mellin remainder, which also samples the odd
cells.  We sharpen the identified channel using the companion-forest layer
structure.

For a finite word \(a=(a_1,\ldots,a_m)\) of nonnegative integers, write
\begin{equation}
 |a|:=m,\qquad
 \|a\|:=\sum_{j=1}^m a_j,\qquad
 \operatorname{enc}(a):=1^{a_1}0\cdots1^{a_m}0,
 \label{eq:degree-word-encoding}
\end{equation}
with \(\operatorname{enc}(\varnothing):=\varnothing\).  Set
\(\psi(0):=\varnothing\).  For \(L\ge1\), define
\begin{equation}
 \eta_L(i)=
 \begin{cases}
  2i-1,&2i\le L+1,\\
  2(L-i+1),&2i>L+1,
 \end{cases}
 \qquad
 \psi(L):=
 \bigl(\eta_L(1)-1,\ldots,\eta_L(L)-1\bigr).
 \label{eq:companion-permutation-explicit}
\end{equation}
Thus \(\psi(L)\) lists the even integers below \(L\) increasingly and then
the odd integers decreasingly.  Extend \(\psi\) to integer words by ordered
concatenation,
\[
 \psi(a_1\cdots a_m):=\psi(a_1)\cdots\psi(a_m),
 \qquad \psi(\varnothing):=\varnothing.
\]
Regarding \(L\) as a one-letter word, put
\begin{equation}
 W_{L,k}:=\psi^k((L)),
 \qquad
 \mathcal L_{R,k}:=W_{0,k}W_{2,k}\cdots W_{2R,k}.
 \label{eq:companion-depth-words}
\end{equation}
Thus \(W_{L,0}=(L)\) and \(W_{L,k}=\varnothing\) for \(k>L\).  For
\(r\ge0\), set \(R:=r+1\) and
\begin{equation}
 C_{r,k}:=\sum_{i=0}^{R}\binom{2i}{k}
 \qquad(k\ge0),
 \label{eq:negative-layer-widths}
\end{equation}
where binomial coefficients outside their natural range are zero.

\begin{proposition}[Exact companion-forest input]
\label{prop:companion-forest-encoding}
For all \(L,k\ge0\),
\begin{equation}
 |W_{L,k}|=\binom Lk,
 \qquad \|W_{L,k}\|=\binom L{k+1}.
 \label{eq:single-root-length-mass}
\end{equation}
Moreover, the complete negative word has the reverse breadth-first
factorization
\begin{equation}
 {
 0N_r=
 \operatorname{enc}(\mathcal L_{R,2R})
 \operatorname{enc}(\mathcal L_{R,2R-1})\cdots
 \operatorname{enc}(\mathcal L_{R,0})1^{R+1}.}
 \label{eq:whole-negative-word-factorization}
\end{equation}
Consequently the depth-\(k\) layer has
\begin{equation}
 |\mathcal L_{R,k}|=C_{r,k},
 \qquad
 \|\mathcal L_{R,k}\|=C_{r,k+1}.
 \label{eq:companion-layer-width-from-forest}
\end{equation}
\end{proposition}

\begin{proof}
Because \(\psi(L)\) is a permutation of \(0,\ldots,L-1\),
\[
 W_{L,k+1}=\prod_{d\in\psi(L)}W_{d,k},
\]
where the product denotes ordered concatenation.  Starting from
\(|W_{L,0}|=1\) and \(\|W_{L,0}\|=L\), induction and the hockey-stick
identity give
\[
 |W_{L,k+1}|=\sum_{d=0}^{L-1}|W_{d,k}|
 =\sum_{d=0}^{L-1}\binom dk=\binom L{k+1},
\]
\[
 \|W_{L,k+1}\|=\sum_{d=0}^{L-1}\|W_{d,k}\|
 =\sum_{d=0}^{L-1}\binom d{k+1}=\binom L{k+2}.
\]
This proves \eqref{eq:single-root-length-mass}, and summation over the even
roots gives \eqref{eq:companion-layer-width-from-forest}.

We now identify the forest encoded by the negative word.  Clo\^{\i}tre
proves that \(0N_r=T(A_r)\), that every orbital core \(A_r\) has property
\(\mathsf S\), and hence that \(T(A_r)\) has the companion property
\(\mathsf R\); see
\cite[Lem.~4.3, Cor.~4.6, Def.~4.4, and Lem.~4.5]{Cloitre2026}.
The root word of \(A_r\) is \((1,3,\ldots,2r+1)\), while the root
calculation in the proof of \cite[Lem.~4.11]{Cloitre2026} shows that the
roots of \(T(A_r)\) have degrees
\[
 (0,2,4,\ldots,2r+2)=(0,2,\ldots,2R);
\]
see also \cite[Sec.~8.3]{Cloitre2026}.  Property \(\mathsf R\) says
precisely that a degree-\(L\) vertex has its children in the order
\(\psi(L)\).  The source-factor argument of
\cite[Lem.~4.9]{Cloitre2026}, with \(\psi\) in place of the
\(\mathsf S\)-word \(\phi\), therefore gives the degree word
\[
 \mathcal L_{R,2R}\mathcal L_{R,2R-1}\cdots\mathcal L_{R,0}.
\]
Finally, the standard degree expansion appends one terminal \(1\) for
each of the \(R+1\) roots.  Applying \(\operatorname{enc}\) to the
successive degree layers proves
\eqref{eq:whole-negative-word-factorization}.
\end{proof}

For $0\le k\le2r+2$, put
\begin{equation}
S_{r,k}:=\sum_{\ell=k}^{2r+2}C_{r,\ell},
\label{eq:negative-layer-tail}
\end{equation}
and set $S_{r,2r+3}=C_{r,2r+3}:=0$.  Define the canonical layer cuts
\begin{equation}
\theta_{r,k}:=2S_{r,k}-C_{r,k}-1.
\label{eq:negative-layer-cuts}
\end{equation}
The endpoint count in \cref{prop:exact-arch-data} gives
\(\mathscr D_r(L_r)=1\).  For cyclic bookkeeping at the bottom layer, we
extend the prefix imbalance by
\begin{equation}
 \mathscr D_r(-1):=\mathscr D_r(L_r)=1.
 \label{eq:cyclic-prefix-convention}
\end{equation}
Then \(\theta_{r,2r+3}=-1\).
By \cref{prop:companion-forest-encoding}, these are exactly the layer widths
of the companion forest encoded by \(N_r\).

\begin{theorem}[Exact negative-layer cuts]
\label{thm:negative-layer-cuts}
For every $0\le k\le2r+3$,
\begin{equation}
{
\mathscr D_r(\theta_{r,k})=1-C_{r,k}.
}
\label{eq:D-layer-exact-new}
\end{equation}
Moreover, for \(0\le k\le2r+2\),
\begin{equation}
\theta_{r,k}-\theta_{r,k+1}=C_{r,k}+C_{r,k+1}.
\label{eq:layer-mesh-new}
\end{equation}
\end{theorem}

\begin{proof}
For \(k=2r+3\), we have \(C_{r,k}=0\), \(\theta_{r,k}=-1\), and
\(\mathscr D_r(\theta_{r,k})=1\) by
\eqref{eq:cyclic-prefix-convention}.  Thus assume \(0\le k\le2r+2\).
Read the reverse breadth-first factorization
\eqref{eq:whole-negative-word-factorization} from left to right.  Layer
$k$ contains $C_{r,k}$
vertices and its children form layer $k+1$, of width $C_{r,k+1}$.  After the
layers \(k,k+1,\ldots,2r+2\) have been read, their number of degree symbols is
$S_{r,k}$, while their total degree is
\[
 \sum_{\ell=k}^{2r+2}C_{r,\ell+1}
 =S_{r,k}-C_{r,k}.
\]
Hence their binary length in the word \(0N_r\) is
$2S_{r,k}-C_{r,k}$.  Removing the initial zero gives
\eqref{eq:negative-layer-cuts}.  At this cut the excess of zeros over ones in
$0N_r$ is $C_{r,k}$, hence the excess of ones over zeros in $N_r$ is
$1-C_{r,k}$, proving \eqref{eq:D-layer-exact-new}.  Subtracting two adjacent
cut positions gives \eqref{eq:layer-mesh-new}.
\end{proof}

Let $\Phi_{\rm N}$ denote the standard normal distribution function.  For
fixed $z\in\mathbb R$, set
\begin{equation}
k_r(z):=
\left\lfloor
r+1+z\sqrt{\frac{r+1}{2}}
\right\rfloor .
\label{eq:negative-central-index}
\end{equation}
A local central limit argument applied to \eqref{eq:negative-layer-widths}
gives the following complete layer skeleton.

\begin{theorem}[Gaussian negative-layer skeleton]
\label{thm:negative-gaussian-skeleton}
Locally uniformly for $z$ in compact subsets of $\mathbb R$,
\begin{align}
C_{r,k_r(z)}
&\sim
\frac{16}{3\sqrt\pi}\,
\frac{4^r}{\sqrt r}e^{-z^2/2},
\label{eq:C-gaussian-new}\\
\frac{\theta_{r,k_r(z)}}{L_r}
&\longrightarrow \Phi_{\rm N}(-z),
\label{eq:theta-gaussian-new}\\
\frac{\sqrt r}{4^r}
\mathscr D_r(\theta_{r,k_r(z)})
&\longrightarrow
-\frac{16}{3\sqrt\pi}e^{-z^2/2}.
\label{eq:D-gaussian-new}
\end{align}
Define \(\Xi_r^{\rm lay}\) to be the continuous piecewise-affine function
on \([0,1]\) whose nodes are
\[
 \left(
 \frac{\theta_{r,k}}{L_r},
 \frac{\sqrt r}{4^r}\mathscr D_r(\theta_{r,k})
 \right)
 \quad(0\le k\le2r+2),
 \qquad
 \left(
 1,\frac{\sqrt r}{4^r}\mathscr D_r(L_r)
 \right).
\]
Then \(\Xi_r^{\rm lay}\) converges locally uniformly on \((0,1)\) and in
\(L^1[0,1]\) to
\begin{equation}
\boxed{
\Xi_{\rm G}(y)
=-\frac{16}{3\sqrt\pi}
\exp\!\left[-\frac12\bigl(\Phi_{\rm N}^{-1}(y)\bigr)^2\right].
}
\label{eq:XiG-new}
\end{equation}
In addition,
\begin{equation}
 {
 \max_{0\le t\le L_r}|\mathscr D_r(t)|
 \sim
 \frac{16}{3\sqrt\pi}\frac{4^r}{\sqrt r}.}
 \label{eq:sharp-D-amplitude}
\end{equation}
\end{theorem}

\begin{proof}
For $k=k_r(z)$, the largest term in $C_{r,k}$ comes from $i=r+1$.
Writing $i=r+1-j$, the local de Moivre--Laplace theorem
\cite{Petrov1975} gives
\[
\binom{2r+2-2j}{k_r(z)}
\sim
\frac{4^{r+1-j}}{\sqrt{\pi r}}e^{-z^2/2}
\]
for each fixed $j$, while the standard central-binomial bound supplies a
geometric majorant $O(4^{-j})$.  Dominated convergence yields
\eqref{eq:C-gaussian-new}.  The same geometric decomposition of the upper
binomial tails yields \eqref{eq:theta-gaussian-new}; then
\eqref{eq:D-layer-exact-new} gives \eqref{eq:D-gaussian-new}.  The layer mesh
is $O(r^{-1/2})$ after normalization, by \eqref{eq:layer-mesh-new} and the
central-binomial bound, so the polygonal skeleton converges locally uniformly;
the uniform bound \eqref{eq:prefix-imbalance-bound} supplies the \(L^1\)
control away from the terminal segment.  That segment has horizontal length
\(O(r/4^r)\) and normalized height \(O(r^{3/2}/4^r)\), so its \(L^1\)-mass is
\(O(r^{5/2}/16^r)\).

It remains to justify the assertion about the full maximum.  For a binary
word \(W\), let \(h_W(t)\) be the excess of zeros over ones in its prefix
of length \(t\).  Then \(\mathscr D_r(t)=-h_{N_r}(t)\).  Put
\[
 H_r:=\max_{0\le t\le|P_r|}h_{P_r}(t).
\]
The exact-fit estimate of \cite[Lemmas~2.9--2.10]{Cloitre2026} gives
\[
 h_{N_r}(t)\ge0\quad(0\le t<L_r),
 \qquad
 \max_{0\le t<L_r}h_{N_r}(t)\le2H_r-2.
\]
Moreover, \cite[Theorem~1.1(ii)]{Cloitre2026} gives
\[
 H_r=2+\sum_{j=1}^{r}\binom{2j+1}{j}
 \sim\frac{8}{3\sqrt\pi}\frac{4^r}{\sqrt r}.
\]
Hence the limsup in \eqref{eq:sharp-D-amplitude} is at most the displayed
constant.  At the canonical central cut \(k_r(0)=r+1\),
\eqref{eq:D-gaussian-new} gives the reverse inequality for the liminf.
This proves \eqref{eq:sharp-D-amplitude}.
\end{proof}

\begin{corollary}[A canonical negative-arch subsequence at the optimal pointwise scale]
\label{cor:canonical-negative-subsequence}
Set
\[
 t_r:=\theta_{r,r+1},
 \qquad
 n_r:=2(v_r+t_r).
\]
Then
\begin{align}
 n_r&\sim32\,4^r,
 \label{eq:canonical-negative-n}\\
 \Qp(n_r)-\frac{n_r}{2}
 &\sim-\frac{16}{3\sqrt\pi}\frac{4^r}{\sqrt r}.
 \label{eq:canonical-negative-error}
\end{align}
Consequently,
\begin{equation}
 \boxed{\displaystyle
 \lim_{r\to\infty}
 \frac{\sqrt{\log n_r}}{n_r}
 \left(\Qp(n_r)-\frac{n_r}{2}\right)
 =-\frac{\sqrt{\log4}}{6\sqrt\pi}.}
 \label{eq:canonical-negative-natural-limit}
\end{equation}
Equivalently, in Clo\^{\i}tre's base-two normalization,
\begin{equation}
 \left(\frac{\Qp(n_r)}{n_r}-\frac12\right)
 \sqrt{\log_2n_r}
 \longrightarrow-\frac{1}{3\sqrt{2\pi}}.
 \label{eq:canonical-negative-base-two-limit}
\end{equation}
In particular,
\begin{align*}
 \liminf_{n\to\infty}
 \frac{\sqrt{\log n}}{n}
 \left(\Qp(n)-\frac n2\right)
 &\le-\frac{\sqrt{\log4}}{6\sqrt\pi},\\
 \limsup_{n\to\infty}
 \frac{\sqrt{\log n}}{n}
 \left|\Qp(n)-\frac n2\right|
 &\ge\frac{\sqrt{\log4}}{6\sqrt\pi}.
\end{align*}
\end{corollary}

\begin{proof}
At \(z=0\), \eqref{eq:negative-central-index} gives
\(k_r(0)=r+1\).  Hence
\cref{thm:negative-gaussian-skeleton} yields
\[
 \frac{t_r}{L_r}\longrightarrow\frac12,
 \qquad
 \mathscr D_r(t_r)
 \sim-\frac{16}{3\sqrt\pi}\frac{4^r}{\sqrt r}.
\]
For all sufficiently large \(r\), this places \(t_r\) strictly inside the
negative arch, so \eqref{eq:prefix-imbalance-Q} applies.  From
\eqref{eq:UV-branches}, \eqref{eq:arch-endpoints}, and
\eqref{eq:arch-coordinates},
\[
 \Qp(2v_r)-v_r
 =2V(v_r)-1-v_r
 =4a_r-1-(4a_r-r-2)=r+1.
\]
Therefore the exact identity
\[
 \Qp(n_r)-\frac{n_r}{2}
 =r+1+\mathscr D_r(t_r)
\]
holds.  Since \(r+1=o(4^r/\sqrt r)\), this proves
\eqref{eq:canonical-negative-error}.  Moreover,
\[
 v_r\sim\frac{32}{3}4^r,
 \qquad
 L_r\sim\frac{32}{3}4^r,
 \qquad
 t_r\sim\frac{16}{3}4^r,
\]
which gives \eqref{eq:canonical-negative-n} and
\(\log n_r=r\log4+O(1)\).  Combining these asymptotics proves
\eqref{eq:canonical-negative-natural-limit}.  The base-two form follows
from \(\log_2n=(\log n)/(\log2)\), and both final bounds follow along the
subsequence \(n_r\).
\end{proof}

\begin{remark}[Relation to Clo\^{\i}tre's global bounds]
Clo\^{\i}tre already established explicit global lower and upper bounds for
the normalized limsup \cite[Thm.~1.1(iii)--(iv)]{Cloitre2026}.  The present
corollary does not improve or determine that global limsup; it instead
identifies a canonical negative-arch subsequence together with its sign,
phase, and exact asymptotic coefficient.
\end{remark}

Pointwise or \(L^1\) convergence of the full raw negative-arch profile
remains open; weak convergence against Lipschitz tests is proved below.

\begin{figure}[htbp]
\centering
\includegraphics[width=0.88\textwidth]{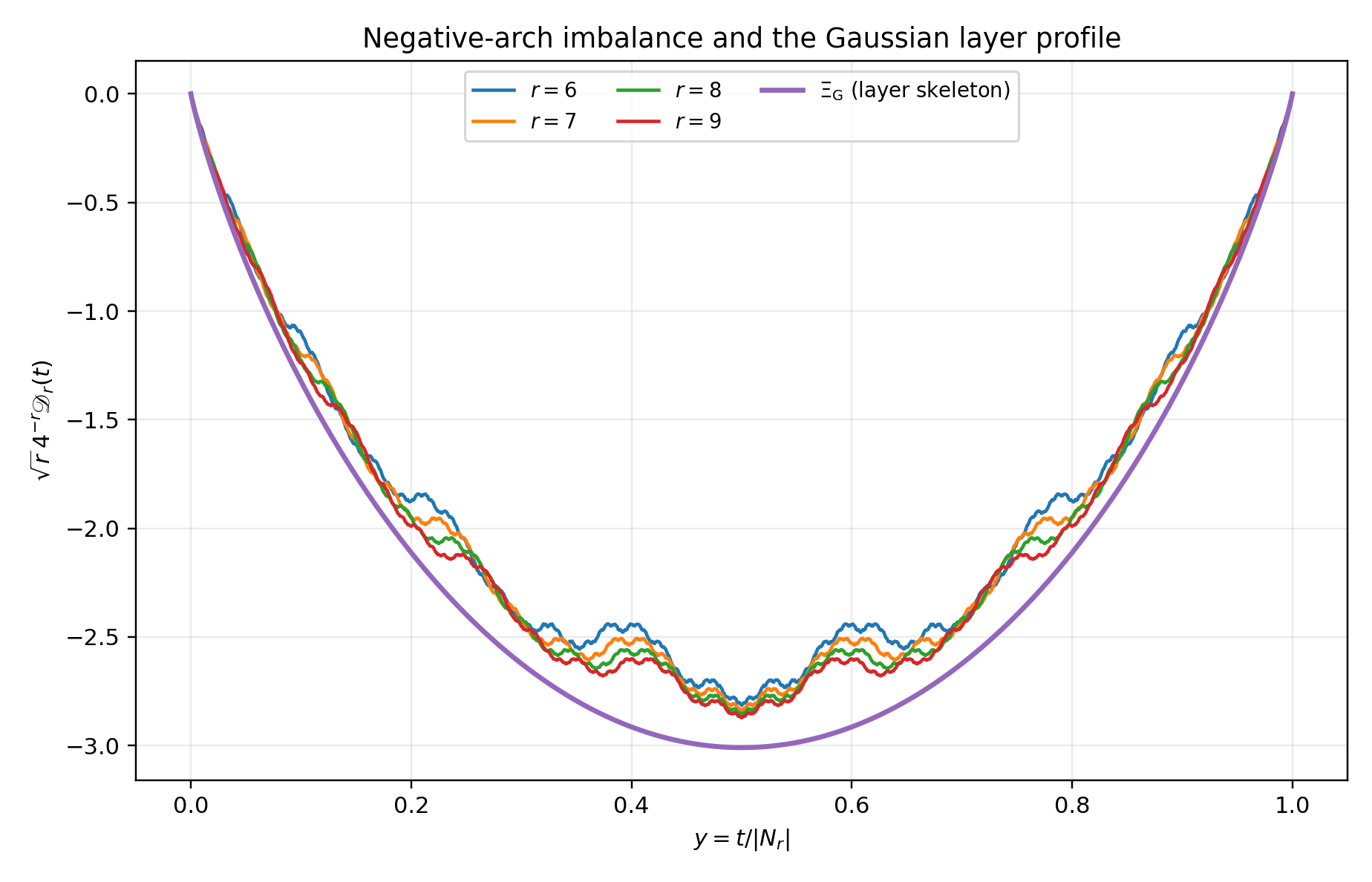}
\caption{Exact normalized negative-arch prefix imbalances and the Gaussian
layer profile.  The layer skeleton and the sharp maximal amplitude are
rigorous; the pointwise full-path convergence visible in the figure is not
asserted.}
\label{fig:negative-gaussian-profile}
\end{figure}

\begin{lemma}[Quadratic layer mass]
\label{lem:quadratic-layer-mass}
One has
\begin{equation}
{
Q_r:=\sum_{k=0}^{2r+2}C_{r,k}^2
\sim
\frac{256}{9\sqrt{2\pi}}\,
\frac{16^r}{\sqrt r}.
}
\label{eq:C-square-asymptotic}
\end{equation}
\end{lemma}

\begin{proof}
Use the central scaling \eqref{eq:negative-central-index}.  By
\eqref{eq:C-gaussian-new},
\[
C_{r,k_r(z)}^2
\sim
\frac{256}{9\pi}\frac{16^r}{r}e^{-z^2}.
\]
One unit in $k$ corresponds to $dz\sim\sqrt{2/r}$.  The central Riemann sum
therefore tends to
\[
\frac{256}{9\pi}\frac{16^r}{r}
\sqrt{\frac r2}\int_{-\infty}^{\infty}e^{-z^2}\,dz
=
\frac{256}{9\sqrt{2\pi}}\frac{16^r}{\sqrt r}.
\]
Standard Gaussian binomial tails make the complement of a fixed central
window uniformly negligible, after which the window is sent to infinity.
\end{proof}

\subsection{Centroid control inside the exact negative layers}

The Gaussian layer nodes do not by themselves determine the area between two
successive cuts.  The missing datum is the order of the degrees inside each
layer.  The companion rule \(\mathsf R\) supplies exactly enough structure to
control this order on the integrated level.

For a finite integer word \(a=(a_1,\ldots,a_n)\), write
\begin{equation}
 |a|=n,\qquad \|a\|=\sum_{j=1}^n a_j,\qquad
 J(a)=\sum_{j=1}^n j a_j,
\end{equation}
and define its centered degree moment
\begin{equation}
 {
 \mathfrak b(a):=J(a)-\frac{n+1}{2}\|a\|.
 }
 \label{eq:centered-degree-bias}
\end{equation}
Thus \(\mathfrak b(a)\) measures the displacement of the degree mass from the
midpoint of the word.

Recall the child order \eqref{eq:companion-permutation-explicit}.  If
\(0\le a<b<L\), then
\begin{equation}
 a\text{ precedes }b\text{ in }\psi(L)
 \quad\Longleftrightarrow\quad a\text{ is even}.
 \label{eq:pair-order-parity}
\end{equation}

Put \(b_{L,k}:=\mathfrak b(W_{L,k})\).  The length and degree mass of
\(W_{L,k}\) are given by \eqref{eq:single-root-length-mass}.

\begin{lemma}[Exact single-root centroid recursion]
\label{lem:single-root-centroid-recursion}
For \(k\ge1\),
\begin{equation}
 {
 b_{L,k}
 =\sum_{d=0}^{L-1}b_{d,k-1}+\varepsilon_{L,k},
 }
 \label{eq:single-root-bias-recursion}
\end{equation}
where
\begin{equation}
 {
 \varepsilon_{L,k}
 =\frac{1}{2k}
 \sum_{0\le a<b<L}
 (-1)^a(b-a)
 \binom{a}{k-1}\binom{b}{k-1}.
 }
 \label{eq:parity-forcing-centroid}
\end{equation}
Moreover, for every compact interval \(I\Subset(0,1)\), uniformly for
\(k/L\in I\),
\begin{equation}
 {
 |b_{L,k}|
 \ll_I
 \frac1L\binom Lk\binom L{k+1}.
 }
 \label{eq:single-root-centroid-bound}
\end{equation}
\end{lemma}

\begin{proof}
For two integer words \(u,v\), direct expansion of
\eqref{eq:centered-degree-bias} gives the concatenation identity
\begin{equation}
 \mathfrak b(uv)
 =\mathfrak b(u)+\mathfrak b(v)
 +\frac12\bigl(|u|\,\|v\|-|v|\,\|u\|\bigr).
 \label{eq:bias-concatenation}
\end{equation}
Since
\[
 W_{L,k}=W_{d_1,k-1}\cdots W_{d_L,k-1},
 \qquad (d_1,\ldots,d_L)=\psi(L),
\]
repeated use of \eqref{eq:bias-concatenation}, together with
\eqref{eq:single-root-length-mass}, yields
\[
 b_{L,k}=\sum_{d<L}b_{d,k-1}
 +\frac12\sum_{i<j}
 \left[
 \binom{d_i}{k-1}\binom{d_j}{k}
 -\binom{d_j}{k-1}\binom{d_i}{k}
 \right].
\]
For an unordered pair \(a<b\), \eqref{eq:pair-order-parity} determines its
sign.  Using
\[
 \binom d k=\binom d{k-1}\frac{d-k+1}{k}
\]
reduces the double sum to \eqref{eq:parity-forcing-centroid}.

We record the ratio estimate leading to
\eqref{eq:single-root-centroid-bound}.  Put \(N=\binom Lk\).  If
\(a=L-1-q\) and \(k/L\in I\), then for some \(\rho_I<1\),
\[
 \binom a{k-1}
 \ll_I \frac{k}{L}N\rho_I^q,
 \qquad
 \sum_{b>a}(b-a)\binom b{k-1}\le qN.
\]
Therefore
\begin{equation}
 |\varepsilon_{L,k}|
 \ll_I \frac{N^2}{L}
 \asymp_I
 \frac1L\binom Lk\binom L{k+1}.
 \label{eq:epsilon-central-bound}
\end{equation}
The summation term in \eqref{eq:single-root-bias-recursion} is stable under
the same estimate.  Indeed, with
\(A_d=\binom d{k-1}\binom d k\),
\[
 \frac{A_{d-1}}{A_d}
 =\frac{(d-k+1)(d-k)}{d^2},
\]
so the mass of \(A_d\) is geometrically concentrated at the upper endpoint
whenever \(k/L\) stays in a compact subinterval of \((0,1)\).  A strong
induction in \(L\), splitting off the exponentially small values of \(d\)
for which \((k-1)/d\) leaves a slightly larger compact interval, gives
\[
 \sum_{d<L}|b_{d,k-1}|
 \ll_I \frac1L\binom Lk\binom L{k+1}.
\]
Together with \eqref{eq:epsilon-central-bound}, this proves
\eqref{eq:single-root-centroid-bound}.
\end{proof}

The complete depth-\(k\) word of the companion forest is the concatenation
of the single-root words \(W_{2i,k}\), \(0\le i\le r+1\).  Denote it by
\(a^{(r,k)}\), and put
\begin{equation}
 \mathfrak b_{r,k}:=\mathfrak b(a^{(r,k)}).
 \label{eq:full-layer-bias}
\end{equation}
Repeated use of \eqref{eq:bias-concatenation} gives the exact decomposition
\begin{equation}
 {
 \mathfrak b_{r,k}
 =\sum_{i=0}^{r+1}b_{2i,k}
 +\frac{1}{2(k+1)}
 \sum_{0\le i<j\le r+1}
 (2j-2i)\binom{2i}{k}\binom{2j}{k}.
 }
 \label{eq:forest-bias-decomposition}
\end{equation}
The second term is particularly transparent: it is only a first moment of
the root separation.

\begin{proposition}[Integrated micro-layer cancellation]
\label{prop:integrated-microlayer-cancellation}
Let
\[
 Q_r:=\sum_{k=0}^{2r+2}C_{r,k}^2.
\]
Then
\begin{equation}
 {
 \sum_{k=0}^{2r+2}|\mathfrak b_{r,k}|=o(Q_r).
 }
 \label{eq:total-layer-bias-small}
\end{equation}
More precisely, on every fixed central Gaussian window
\(|k-(r+1)|\le M\sqrt r\),
\begin{equation}
 |\mathfrak b_{r,k}|
 \ll_M \frac1r C_{r,k}C_{r,k+1}.
 \label{eq:central-forest-bias}
\end{equation}
\end{proposition}

\begin{proof}
On the stated central window, the contribution of the top roots
\(2r+2,2r,\ldots\) decreases geometrically: for fixed root deficit,
\[
 \frac{\binom{2i-2}{k}}{\binom{2i}{k}}
 =\frac{(2i-k)(2i-k-1)}{(2i)(2i-1)}
 =\frac14+O_M(r^{-1/2}).
\]
Thus \(C_{r,k}\) and \(C_{r,k+1}\) are comparable with their top-root
terms, and the root-deficit distribution has a uniformly bounded first
moment.  Applying \eqref{eq:single-root-centroid-bound} to the internal
terms in \eqref{eq:forest-bias-decomposition} gives
\[
 \sum_i|b_{2i,k}|
 \ll_M r^{-1}C_{r,k}C_{r,k+1}.
\]
The exact cross term in \eqref{eq:forest-bias-decomposition} has the same
bound because \(k+1\asymp r\) and the weighted mean root separation is
bounded.  This proves \eqref{eq:central-forest-bias}.

For the tails use the trivial centroid estimate
\(|\mathfrak b_{r,k}|\le \tfrac12 C_{r,k}C_{r,k+1}\).  The Gaussian tail
bound underlying \cref{lem:quadratic-layer-mass} implies that, after first
letting \(r\to\infty\) and then \(M\to\infty\), the contribution of
\(|k-(r+1)|>M\sqrt r\) is arbitrarily small compared with \(Q_r\).
Together with \eqref{eq:central-forest-bias}, this proves
\eqref{eq:total-layer-bias-small}.
\end{proof}

There is also an exact geometric interpretation of
\(\mathfrak b_{r,k}\).  Put
\[
 n=C_{r,k},\qquad s=C_{r,k+1},\qquad \ell=n+s.
\]
The binary factor of layer \(k\) is
\(1^{a_1}0\cdots1^{a_n}0\), where
\(a=a^{(r,k)}\) and \(\sum a_j=s\).  Summing the negative imbalance
\(-\mathscr D_r\) over this factor gives the following identity.

\begin{lemma}[Exact layer-area identity]
\label{lem:exact-layer-area}
For \(0\le k\le2r+2\), the segment between \(\theta_{r,k+1}\) and
\(\theta_{r,k}\) satisfies
\begin{equation}
 {
 -\sum_{t=\theta_{r,k+1}}^{\theta_{r,k}-1}\mathscr D_r(t)
 =\frac{\ell^2-3\ell}{2}+2\mathfrak b_{r,k}.
 }
 \label{eq:exact-layer-area}
\end{equation}
If \(q_{r,k}\) is the affine chord joining the two exact endpoint heights,
then the difference between the raw discrete area and the discrete chord
area is exactly
\begin{equation}
 {
 \sum_{t=\theta_{r,k+1}}^{\theta_{r,k}-1}
 \bigl(-\mathscr D_r(t)-q_{r,k}(t)\bigr)
 =2\mathfrak b_{r,k}-C_{r,k+1}.
 }
 \label{eq:raw-chord-area-error}
\end{equation}
\end{lemma}

\begin{proof}
For \(k=2r+2\), the child width is zero and the layer factor is the initial
zero in \(0N_r\); the cyclic value \eqref{eq:cyclic-prefix-convention} makes
the same calculation below valid at this boundary.
The zero positions inside the factor are
\[
 p_j=j+\sum_{i\le j}a_i,
 \qquad 1\le j\le n.
\]
The starting value of \(-\mathscr D_r\) is \(s-1\).  Summing the effect of
each bit on all later prefixes and using
\[
 \sum_{j=1}^np_j
 =\frac{n(n+1)}2+(n+1)s-J(a)
\]
gives \eqref{eq:exact-layer-area} after inserting
\eqref{eq:centered-degree-bias}.  The discrete affine chord has endpoint
values \(s-1\) and \(n-1\); subtracting its arithmetic-progression sum from
\eqref{eq:exact-layer-area} gives \eqref{eq:raw-chord-area-error}.
\end{proof}

\subsection{Weak Gaussian convergence and the exact negative-arch area}

Define the raw normalized step profile
\begin{equation}
 \Xi_r^{\rm raw}(y)
 :=\frac{\sqrt r}{4^r}
 \mathscr D_r(\lfloor L_r y\rfloor),
 \qquad 0\le y<1,
 \label{eq:raw-profile-normalized}
\end{equation}
The polygonal profile \(\Xi_r^{\rm lay}\), including its terminal segment,
was defined in \cref{thm:negative-gaussian-skeleton}.  Full pointwise or
\(L^1\) convergence of
\(\Xi_r^{\rm raw}\) is still not known.  For the Mellin problem, however,
a weaker statement is sufficient and can now be proved.

\begin{theorem}[Weak Gaussian convergence of the negative-arch profile]
\label{thm:weak-gaussian-full-profile}
For every Lipschitz function \(\phi:[0,1]\to\mathbb C\),
\begin{equation}
 \boxed{
 \frac1{L_r}\sum_{t=0}^{L_r-1}
 \Xi_r^{\rm raw}\!\left(\frac{t}{L_r}\right)
 \phi\!\left(\frac{t}{L_r}\right)
 \longrightarrow
 \int_0^1\Xi_{\rm G}(y)\phi(y)\,dy.
 }
 \label{eq:weak-gaussian-test}
\end{equation}
For every \(M<\infty\), the convergence is uniform over all Lipschitz
functions \(\phi\) satisfying
\(\|\phi\|_\infty+\operatorname{Lip}(\phi)\le M\).
\end{theorem}

\begin{proof}
Put \(R:=r+1\), \(B:=\|\phi\|_\infty\), and
\(K:=\operatorname{Lip}(\phi)\).  By
\cref{thm:negative-gaussian-skeleton},
\[
 \int_0^1\Xi_r^{\rm lay}(y)\phi(y)\,dy
 \longrightarrow
 \int_0^1\Xi_{\rm G}(y)\phi(y)\,dy.
\]
It remains to compare the raw sum with the integral on the left.

The cyclic interval corresponding to \(k=2R\) is
\([\theta_{r,2R+1},\theta_{r,2R})=[-1,0)\).  It contains only the artificial
value from \eqref{eq:cyclic-prefix-convention} and is part of neither the
raw sum nor \(\Xi_r^{\rm lay}\).  The genuine layer intervals
\[
 [\theta_{r,k+1},\theta_{r,k})
 \qquad(0\le k\le2R-1)
\]
cover \(0\le t<\theta_{r,0}\).

We first dispose of the remaining terminal interval.  By
\eqref{eq:D-layer-exact-new},
\(\mathscr D_r(\theta_{r,0})=1-C_{r,0}=-R\), and the terminal factor
\(1^{R+1}\) in \eqref{eq:whole-negative-word-factorization} gives
\[
 \mathscr D_r(\theta_{r,0}+j)=-R+j
 \qquad(0\le j\le R).
\]
Consequently
\[
 \sum_{t=\theta_{r,0}}^{L_r-1}|\mathscr D_r(t)|
 =\frac{R(R+1)}2.
\]
Since \(L_r\asymp4^r\), the normalized raw contribution of this interval is
\[
 O\!\left(B\frac{r^{5/2}}{16^r}\right)=o(1).
\]
The integral of \(|\Xi_r^{\rm lay}\phi|\) over the same interval obeys the
same bound.

On a genuine layer interval, let \(d_{r,k}\) be the affine chord of
\(\mathscr D_r\), and freeze \(\phi\) at one point of the interval.  The
signed area error is the negative of
\eqref{eq:raw-chord-area-error}.  Hence
\eqref{eq:total-layer-bias-small} gives
\[
 \sum_{k=0}^{2R-1}
 O\!\left(B\bigl(|\mathfrak b_{r,k}|+C_{r,k+1}\bigr)\right)
 =o(BQ_r),
\]
because \(\sum_kC_{r,k}=O(4^r)\), whereas
\(Q_r\asymp16^r/\sqrt r\) by
\cref{lem:quadratic-layer-mass}.

Put \(\ell_k:=C_{r,k}+C_{r,k+1}\).  The raw and chord heights are both
\(O(4^r/\sqrt r)\), by \eqref{eq:prefix-imbalance-bound}, while the
variation of \(\phi(t/L_r)\) across the \(k\)-th interval is at most
\(K\ell_k/L_r\).  Thus the total variation error before normalization is
\[
 O\!\left(
 K\frac{4^r}{\sqrt r\,L_r}\sum_k\ell_k^2
 \right)
 =O\!\left(K\frac{16^r}{r}\right),
\]
since \(\sum_k\ell_k^2=O(Q_r)\).  Multiplication by
\(\sqrt r/(4^rL_r)\asymp\sqrt r/16^r\) makes this
\(O(Kr^{-1/2})\), while the frozen-weight error becomes \(o(B)\).

Finally, the ordinary left Riemann sum of the chord profile may be replaced
by its integral.  Indeed, before multiplication by
\(\sqrt r/(4^rL_r)\), the error is
\[
 O\!\left((Br+K)\frac{4^r}{\sqrt r}\right):
\]
there are \(2R\) affine pieces, their heights are
\(O(4^r/\sqrt r)\), and the total variation of
\(\phi(\,\cdot\,/L_r)\) is at most \(K\).  After normalization this is
\(O((Br+K)/L_r)=o(1)\).  All bounds are uniform when \(B+K\le M\), which
proves \eqref{eq:weak-gaussian-test}.
\end{proof}

The constant test function yields the leading integrated asymptotic.
Put
\begin{equation}
 \mathcal A_r:=-\sum_{t=0}^{L_r-1}\mathscr D_r(t).
 \label{eq:negative-arch-area-definition}
\end{equation}

\begin{corollary}[Exact Gaussian negative-arch area]
\label{cor:exact-negative-area}
The full raw negative arch satisfies
\begin{equation}
 {
 \mathcal A_r
 \sim
 \frac{512}{9\sqrt{2\pi}}\,
 \frac{16^r}{\sqrt r}.
 }
 \label{eq:exact-negative-area-asymptotic}
\end{equation}
Equivalently, with
\(Q_r=\sum_kC_{r,k}^2\),
\begin{equation}
 {\mathcal A_r\sim2Q_r.}
 \label{eq:area-twice-Q}
\end{equation}
\end{corollary}

\begin{proof}
Take \(\phi\equiv1\) in \cref{thm:weak-gaussian-full-profile}.  Since
\[
 L_r\sim\frac{32}{3}4^r
\]
and, after the substitution \(y=\Phi_{\rm N}(z)\),
\[
 -\int_0^1\Xi_{\rm G}(y)\,dy
 =\frac{16}{3\sqrt{2\pi}},
\]
we obtain \eqref{eq:exact-negative-area-asymptotic}.  Combining it with
\eqref{eq:C-square-asymptotic} gives \eqref{eq:area-twice-Q}.
\end{proof}

For reference, exact degree-word dynamic programming gives the following
high-rank values without constructing the exponentially long binary words:
\begin{center}
\begin{tabular}{@{}r r r@{}}
\toprule
$r$&$\mathcal A_r/Q_r$&$\sqrt r\,\mathcal A_r/16^r$\\
\midrule
10&2.005255&21.052449\\
20&2.002702&21.826865\\
30&2.001818&22.105004\\
40&2.001370&22.248196\\
50&2.001099&22.335482\\
60&2.000917&22.394257\\
\bottomrule
\end{tabular}
\end{center}
The limiting constant in the last column is
\(512/(9\sqrt{2\pi})=22.695383\ldots\).  These computations are checks of
the theorem, not inputs to its proof.

\subsection{The first integrated Edgeworth correction}
\label{subsec:first-integrated-edgeworth}

The preceding theorem determines the leading negative-arch mass.  The next
coefficient can also be obtained rigorously, without assuming pointwise
convergence of the raw profile.  This is the first place where the
second-order polylogarithmic hierarchy can be seen unconditionally.

Put
\[
 R:=r+1,
 \qquad
 \mathscr C_R:=\binom{4R}{2R},
\]
and abbreviate
\begin{equation}
 Q_r:=\sum_{k=0}^{2R} C_{r,k}^2,
 \qquad
 P_r:=\sum_{k=0}^{2R-1}C_{r,k}C_{r,k+1},
 \qquad
 \mathfrak B_r:=\sum_{k=0}^{2R}\mathfrak b_{r,k}.
 \label{eq:QPB-def}
\end{equation}
The first two quantities depend only on the companion-layer widths, whereas
\(\mathfrak B_r\) retains the integrated order information inside the
layers.

\begin{lemma}[Corner asymptotics for the quadratic layer sums]
\label{lem:QP-corner-asymptotics}
As \(R\to\infty\),
\begin{align}
 Q_r
 &=\mathscr C_R
 \left(
 \frac{16}{9}-\frac{40}{81R}+O(R^{-2})
 \right),
 \label{eq:Q-corner-expansion}\\
 P_r
 &=\mathscr C_R
 \left(
 \frac{16}{9}-\frac{112}{81R}+O(R^{-2})
 \right).
 \label{eq:P-corner-expansion}
\end{align}
Consequently,
\begin{equation}
 \frac{P_r}{Q_r}
 =1-\frac{1}{2R}+O(R^{-2}).
 \label{eq:P-over-Q-expansion}
\end{equation}
\end{lemma}

\begin{proof}
Vandermonde's identity gives
\begin{align*}
 Q_r
 &=\sum_{0\le i,j\le R}
 \binom{2i+2j}{2i},\\
 P_r
 &=\sum_{0\le i,j\le R}
 \binom{2i+2j}{2i+1}.
\end{align*}
Both sums are exponentially concentrated at the corner \((i,j)=(R,R)\).
Write \(i=R-p\), \(j=R-q\).  Uniformly for
\(p+q=o(R^{1/2})\), elementary factorial-ratio expansion gives
\begin{equation}
 \frac{\binom{4R-2p-2q}{2R-2p}}{\mathscr C_R}
 =4^{-p-q}
 \left(1+\frac{\alpha_{p,q}}{R}
 +O\!\left(\frac{(1+p+q)^4}{R^2}\right)\right),
 \label{eq:Q-corner-local}
\end{equation}
where
\[
 \alpha_{p,q}
 =-\frac{p^2+q^2}{2}+pq+\frac{p+q}{4}.
\]
Similarly,
\begin{equation}
 \frac{\binom{4R-2p-2q}{2R-2p+1}}{\mathscr C_R}
 =4^{-p-q}
 \left(1+\frac{\beta_{p,q}}{R}
 +O\!\left(\frac{(1+p+q)^4}{R^2}\right)\right),
 \label{eq:P-corner-local}
\end{equation}
with
\[
 \beta_{p,q}
 =-\frac{p^2+q^2}{2}+pq+\frac{5p-3q}{4}-\frac12.
\]
The complement of, say, \(p+q\le R^{1/3}\) is exponentially small relative
to \(\mathscr C_R\), so the expansions can be summed termwise.  The needed
geometric sums are
\begin{align*}
 \sum_{p,q\ge0}4^{-p-q}&=\frac{16}{9},\\
 \sum_{p,q\ge0}4^{-p-q}\alpha_{p,q}&=-\frac{40}{81},\\
 \sum_{p,q\ge0}4^{-p-q}\beta_{p,q}&=-\frac{112}{81}.
\end{align*}
This proves \eqref{eq:Q-corner-expansion}--\eqref{eq:P-corner-expansion}, and
the quotient gives \eqref{eq:P-over-Q-expansion}.
\end{proof}

The remaining contribution comes from the centered degree moments.  For a
single root put
\begin{equation}
 \beta_L:=\sum_{k=0}^{L}b_{L,k}.
 \label{eq:beta-single-total}
\end{equation}
Summing the exact recursion in
\cref{lem:single-root-centroid-recursion} over the depth gives a closed
one-dimensional recurrence.

\begin{lemma}[Summed single-root centroid]
\label{lem:summed-single-root-centroid}
For \(L\ge1\),
\begin{equation}
 \beta_L=\sum_{d=0}^{L-1}\beta_d+e_L,
 \label{eq:beta-summed-recurrence}
\end{equation}
where
\begin{equation}
 e_L
 =\frac12\sum_{0\le a<b<L}
 (-1)^a(b-a)
 \frac{(a+b+1)!}{(a+1)!(b+1)!}.
 \label{eq:eL-exact}
\end{equation}
Moreover,
\begin{align}
 e_L
 &=(-1)^L\frac{2}{45L}\binom{2L}{L}
 \left(1+O(L^{-1})\right),
 \label{eq:eL-asymptotic}\\
 {
 \beta_L
 =(-1)^L\frac{1}{27L}\binom{2L}{L}
 \left(1+O(L^{-1})\right).}
 \label{eq:betaL-asymptotic}
\end{align}
\end{lemma}

\begin{proof}
Summing \eqref{eq:parity-forcing-centroid} over \(k\) uses the exact identity
\[
 \sum_{k\ge1}\frac1k
 \binom{a}{k-1}\binom{b}{k-1}
 =\frac{(a+b+1)!}{(a+1)!(b+1)!},
\]
which yields \eqref{eq:beta-summed-recurrence} and
\eqref{eq:eL-exact}.

For \eqref{eq:eL-asymptotic}, set
\(a=L-1-p\), \(b=L-1-q\), so that \(p>q\ge0\).  Relative to the central
binomial coefficient,
\[
 \frac{(a+b+1)!}{(a+1)!(b+1)!}
 =\binom{2L}{L}
 \frac{2^{-p-q-1}}{L}
 \left(1+O\!\left(\frac{(1+p+q)^2}{L}\right)\right)
\]
on every fixed corner window, while the remaining terms have a uniformly
geometric tail.  Hence
\[
 \frac{Le_L}{(-1)^{L-1}\binom{2L}{L}}
 \longrightarrow
 \sum_{p>q\ge0}(-1)^p(p-q)2^{-p-q-2}
 =-\frac{2}{45},
\]
which proves \eqref{eq:eL-asymptotic} with the stated error.

Let \(S_L=\sum_{d\le L}\beta_d\).  Then
\(S_L=2S_{L-1}+e_L\), and therefore
\[
 \beta_L
 =e_L+\sum_{j=1}^{L-1}2^{j-1}e_{L-j}.
\]
For fixed \(j\),
\[
 \frac{e_{L-j}}{e_L}
 =(-1)^j4^{-j}\left(1+O(j/L)\right),
\]
and the tail is again geometric.  Thus
\[
 \frac{\beta_L}{e_L}
 \longrightarrow
 1+\sum_{j\ge1}2^{j-1}(-1)^j4^{-j}
 =1-\frac16=\frac56.
\]
Combining this with \eqref{eq:eL-asymptotic} gives
\eqref{eq:betaL-asymptotic}.
\end{proof}

The forest bias has two parts.  Summing
\eqref{eq:forest-bias-decomposition} over \(k\) gives exactly
\begin{equation}
 \mathfrak B_r
 =\mathfrak B_r^{\rm int}+\mathfrak B_r^{\rm cross},
 \label{eq:B-split}
\end{equation}
where
\begin{align}
 \mathfrak B_r^{\rm int}
 &=\sum_{i=0}^{R}\beta_{2i},
 \label{eq:B-int-exact}\\
 \mathfrak B_r^{\rm cross}
 &=\sum_{0\le i<j\le R}(j-i)
 \frac{(2i+2j+1)!}{(2i+1)!(2j+1)!}.
 \label{eq:B-cross-exact}
\end{align}

\begin{proposition}[First-order forest centroid constant]
\label{prop:forest-centroid-first-order}
As \(R\to\infty\),
\begin{align}
 \mathfrak B_r^{\rm int}
 &=\frac{8}{405R}\mathscr C_R
 +O\!\left(\frac{\mathscr C_R}{R^2}\right),
 \label{eq:B-int-asymptotic}\\
 \mathfrak B_r^{\rm cross}
 &=\frac{64}{135R}\mathscr C_R
 +O\!\left(\frac{\mathscr C_R}{R^2}\right),
 \label{eq:B-cross-asymptotic}\\
 {
 \mathfrak B_r
 =\frac{40}{81R}\mathscr C_R
 +O\!\left(\frac{\mathscr C_R}{R^2}\right).}
 \label{eq:B-total-asymptotic}
\end{align}
Equivalently,
\begin{equation}
 \frac{\mathfrak B_r}{Q_r}
 =\frac{5}{18R}+O(R^{-2}).
 \label{eq:B-over-Q-asymptotic}
\end{equation}
\end{proposition}

\begin{proof}
Since only even roots occur in \eqref{eq:B-int-exact},
\cref{lem:summed-single-root-centroid} gives
\[
 \beta_{2i}
 =\frac{1}{54i}\binom{4i}{2i}
 \left(1+O(i^{-1})\right).
\]
The final roots dominate geometrically with ratio \(1/16\), hence
\[
 \mathfrak B_r^{\rm int}
 =\frac{\mathscr C_R}{54R}
 \left(\sum_{j\ge0}16^{-j}+O(R^{-1})\right)
 =\frac{8}{405R}\mathscr C_R
 +O(\mathscr C_RR^{-2}).
\]

For the cross term, write \(i=R-p\), \(j=R-q\), with \(p>q\ge0\).  A
factorial-ratio expansion gives
\[
 \frac{(2i+2j+1)!}{(2i+1)!(2j+1)!}
 =\frac{\mathscr C_R}{R}
 4^{-p-q}
 \left(1+O\!\left(\frac{(1+p+q)^2}{R}\right)\right),
\]
again with a geometric tail.  Therefore
\[
 \frac{R\mathfrak B_r^{\rm cross}}{\mathscr C_R}
 \longrightarrow
 \sum_{p>q\ge0}(p-q)4^{-p-q}
 =\frac{64}{135}.
\]
Adding the two pieces gives \eqref{eq:B-total-asymptotic}; division by
\eqref{eq:Q-corner-expansion} gives \eqref{eq:B-over-Q-asymptotic}.
\end{proof}

\begin{proposition}[Exact total-area decomposition]
\label{prop:exact-total-area-decomposition}
With \(R=r+1\) and the notation in \eqref{eq:QPB-def}, put
\begin{equation}
 M_r:=\sum_{k=0}^{2R}C_{r,k}
 =\sum_{i=0}^{R}4^i
 =\frac{4^{R+1}-1}{3}.
 \label{eq:Mr-layer-mass}
\end{equation}
Then the complete raw negative-arch area satisfies the finite identity
\begin{equation}
 {
 \mathcal A_r
 =Q_r+P_r+2\mathfrak B_r-3M_r+R+2.}
 \label{eq:A-QPB-exact}
\end{equation}
\end{proposition}

\begin{proof}
Set \(C_{r,2R+1}=0\) and
\(\ell_{r,k}=C_{r,k}+C_{r,k+1}\).  Summing
\eqref{eq:exact-layer-area} for \(0\le k\le2R\) gives
\begin{align}
 \mathcal L_r
 &:=-\sum_{k=0}^{2R}
 \sum_{t=\theta_{r,k+1}}^{\theta_{r,k}-1}\mathscr D_r(t)
 \nonumber\\
 &=\frac12\sum_{k=0}^{2R}\ell_{r,k}^2
 -\frac32\sum_{k=0}^{2R}\ell_{r,k}
 +2\mathfrak B_r\nonumber\\
 &=Q_r+P_r+2\mathfrak B_r-3M_r
 +\frac{3C_{r,0}-C_{r,0}^2}{2}.
 \label{eq:summed-layer-area-before-tail}
\end{align}
Here we used
\[
 \sum_{k=0}^{2R}\ell_{r,k}^2
 =2Q_r-C_{r,0}^2+2P_r,
 \qquad
 \sum_{k=0}^{2R}\ell_{r,k}=2M_r-C_{r,0}.
\]
The cyclic convention \eqref{eq:cyclic-prefix-convention} means that the
left side of \eqref{eq:summed-layer-area-before-tail} contains the single
extra value \(-\mathscr D_r(-1)=-1\).  Moreover,
\[
 C_{r,0}=R+1,
 \qquad
 \theta_{r,0}=L_r-R-1.
\]
The final factor \(1^{R+1}\) in
\eqref{eq:whole-negative-word-factorization} therefore gives
\begin{equation}
 \mathscr D_r(\theta_{r,0}+j)=-R+j
 \qquad(0\le j\le R),
 \label{eq:terminal-tail-heights}
\end{equation}
and hence
\[
 -\sum_{t=\theta_{r,0}}^{L_r-1}\mathscr D_r(t)
 =\frac{R(R+1)}2.
\]
Thus \(\mathcal A_r=\mathcal L_r+1+R(R+1)/2\).  Substituting
\(C_{r,0}=R+1\) into \eqref{eq:summed-layer-area-before-tail} yields
\eqref{eq:A-QPB-exact}.  Finally,
\eqref{eq:Mr-layer-mass} follows by summing the binomial theorem over the
even roots.
\end{proof}

We can now sharpen \cref{cor:exact-negative-area} by one full inverse-rank
order.

\begin{theorem}[First Edgeworth correction to the negative-arch area]
\label{thm:first-edgeworth-negative-area}
The exact raw negative-arch mass satisfies
\begin{equation}
 \boxed{
 \mathcal A_r
 =\frac{512}{9\sqrt{2\pi}}\,
 \frac{16^r}{\sqrt r}
 \left(
 1-\frac{13}{16r}+O(r^{-2})
 \right).}
 \label{eq:A-first-edgeworth}
\end{equation}
Moreover,
\begin{equation}
 {
 Q_r
 =\frac{256}{9\sqrt{2\pi}}\,
 \frac{16^r}{\sqrt r}
 \left(
 1-\frac{121}{144r}+O(r^{-2})
 \right),}
 \label{eq:Q-first-edgeworth}
\end{equation}
and
\begin{equation}
 {
 \frac{\mathcal A_r}{Q_r}
 =2+\frac{1}{18r}+O(r^{-2}).}
 \label{eq:A-over-Q-first-edgeworth}
\end{equation}
\end{theorem}

\begin{proof}
Apply \cref{prop:exact-total-area-decomposition}.
The last two terms are exponentially smaller than
\(\mathscr C_R/R^2\).  Substituting
\cref{lem:QP-corner-asymptotics,prop:forest-centroid-first-order} into
\eqref{eq:A-QPB-exact} yields
\begin{equation}
 \mathcal A_r
 =\mathscr C_R
 \left(
 \frac{32}{9}-\frac{8}{9R}+O(R^{-2})
 \right)
 =\frac{32}{9}\mathscr C_R
 \left(1-\frac{1}{4R}+O(R^{-2})\right).
 \label{eq:A-central-binomial-form}
\end{equation}
Finally,
\[
 \mathscr C_R
 =\frac{16^R}{\sqrt{2\pi R}}
 \left(1-\frac{1}{16R}+O(R^{-2})\right),
 \qquad R=r+1.
\]
Converting from \(R\) to \(r\) gives
\eqref{eq:A-first-edgeworth}.  The same calculation applied to
\eqref{eq:Q-corner-expansion} gives \eqref{eq:Q-first-edgeworth}, and the
quotient gives \eqref{eq:A-over-Q-first-edgeworth}.
\end{proof}

This theorem provides a rigorous second term in the channel relevant to the
constant test function.  In particular it already produces the next
polylogarithm in an exact scale generating function.

\begin{corollary}[Two-term polylogarithmic skeleton for the flat arch mass]
\label{cor:flat-area-polylog}
Let
\begin{equation}
 \mathscr F_{\!A}(z)
 :=\sum_{r\ge1}\frac{\mathcal A_r}{16^r}z^r,
 \qquad |z|<1,
 \label{eq:flat-area-generating}
\end{equation}
and put
\[
 a_0:=\frac{512}{9\sqrt{2\pi}}.
\]
Then
\begin{equation}
 {
 \mathscr F_{\!A}(z)
 =a_0\operatorname{Li}_{1/2}(z)
 -\frac{13a_0}{16}\operatorname{Li}_{3/2}(z)
 +\mathscr E_A(z),}
 \label{eq:flat-area-two-polylogs}
\end{equation}
where the Taylor coefficients of \(\mathscr E_A\) are
\(O(r^{-5/2})\).  Consequently \(\mathscr E_A\) and its first derivative
extend continuously to \(|z|\le1\).
\end{corollary}

\begin{proof}
Divide \eqref{eq:A-first-edgeworth} by \(16^r\) and sum termwise.  The two
explicit coefficient sequences are \(r^{-1/2}\) and \(r^{-3/2}\), giving
the two displayed polylogarithms.  The remaining coefficients are
\(O(r^{-5/2})\), whose first differentiated series is absolutely summable
on the unit circle.
\end{proof}

\section{Outlook: from one arch channel to the full Mellin remainder}
\label{sec:outlook}

The present paper reaches a natural stopping point.  For the complete
$t=-1$ slice, the order-$X\log X$ term, the continuous log-periodic
order-$X$ fluctuation, and the induced boundary resonance lattice are
explicit; for the negative-even arch channel, the next square-root layer
has a rigorous weak Gaussian limit and integrated Edgeworth corrections.
The present analysis controls the negative-even channel only.  A
continuation across \(\RePart w=0\) requires a decomposition of the complete
Mellin remainder together with uniform analytic control of the remaining
cells.  This lies beyond the scope of the present paper.

No finite Edgeworth tower can replace that resummation: the coefficient
$1547/1536$ in Appendix~A records a further exact constant test, not a
full-slice Puiseux expansion.  Accordingly, this paper makes no claim of a
full $w$-dependent continuation into $\RePart w<0$.  It supplies the exact
full-slice boundary skeleton and the rigorous Gaussian arch model against
which such a continuation theorem must be measured.

\appendix

\section{The second integrated Edgeworth coefficient}
\label{subsec:second-flat-edgeworth}

The main text stops after the first inverse-rank correction.  For completeness, we push the constant test one order further.  The exact area identity \eqref{eq:A-QPB-exact} reduces the calculation to corner expansions of hypergeometric sums and the summed forest centroid.

\begin{lemma}[Second corner coefficients]
\label{lem:second-corner-coefficients}
With $R=r+1$ and $\mathscr C_R=\binom{4R}{2R}$, one has
\begin{align}
 Q_r
 &=\mathscr C_R\left(
 \frac{16}{9}-\frac{40}{81R}+\frac{61}{243R^2}
 +O(R^{-3})\right),
 \label{eq:Q-second-corner}\\
 P_r
 &=\mathscr C_R\left(
 \frac{16}{9}-\frac{112}{81R}+\frac{349}{243R^2}
 +O(R^{-3})\right),
 \label{eq:P-second-corner}\\
 \mathfrak B_r^{\rm int}
 &=\mathscr C_R\left(
 \frac{8}{405R}+\frac{2}{6075R^2}+O(R^{-3})\right),
 \label{eq:Bint-second-corner}\\
 \mathfrak B_r^{\rm cross}
 &=\mathscr C_R\left(
 \frac{64}{135R}-\frac{1184}{2025R^2}+O(R^{-3})\right).
 \label{eq:Bcross-second-corner}
\end{align}
Consequently
\begin{equation}
 {
 \mathfrak B_r
 =\mathscr C_R\left(
 \frac{40}{81R}-\frac{142}{243R^2}+O(R^{-3})\right).}
 \label{eq:B-second-corner}
\end{equation}
\end{lemma}

\begin{proof}
Put \(s=p+q\), \(d=p-q\), and let
\(\mathscr C_R=\binom{4R}{2R}\).  The exact product representation for the
first corner ratio is
\begin{equation}
 \frac{\binom{4R-2s}{2R-2p}}{\mathscr C_R}
 =4^{-s}
 \frac{
   \displaystyle\prod_{a=0}^{2p-1}\left(1-\frac{a}{2R}\right)
   \displaystyle\prod_{b=0}^{2q-1}\left(1-\frac{b}{2R}\right)}
  {\displaystyle\prod_{c=0}^{2s-1}\left(1-\frac{c}{4R}\right)}.
 \label{eq:Q-corner-product}
\end{equation}
Writing
\(S_2(m)=\sum_{j=0}^{m-1}j^2=m(m-1)(2m-1)/6\), logarithmic expansion of
\eqref{eq:Q-corner-product} gives
\begin{equation}
 \log\!\left(
 4^s\frac{\binom{4R-2s}{2R-2p}}{\mathscr C_R}\right)
 =\frac{\alpha_{p,q}}R+\frac{\lambda_{p,q}}{R^2}
 +O\!\left(\frac{(1+s)^4}{R^3}\right),
 \label{eq:Q-corner-log-second}
\end{equation}
where
\begin{equation}
 \alpha_{p,q}=-\frac{d^2}{2}+\frac{s}{4},
 \qquad
 \lambda_{p,q}
 =-\frac{S_2(2p)+S_2(2q)}8+\frac{S_2(2s)}{32}
 =-\frac{(2s-1)(4d^2-s)}{32}.
 \label{eq:alpha-lambda-corner}
\end{equation}
Exponentiating yields
\begin{equation}
 \frac{\binom{4R-2s}{2R-2p}}{\mathscr C_R}
 =4^{-s}\left(
 1+\frac{\alpha_{p,q}}R
 +\frac{\alpha^{(2)}_{p,q}}{R^2}
 +O\!\left(\frac{(1+s)^6}{R^3}\right)\right),
 \label{eq:Q-corner-second-local}
\end{equation}
with the explicit polynomial
\begin{equation}
 {
 \alpha^{(2)}_{p,q}
 =\frac{d^4}{8}-\frac{3sd^2}{8}
 +\frac{d^2}{8}+\frac{3s^2}{32}-\frac{s}{32}.}
 \label{eq:alpha-second-polynomial}
\end{equation}

The second corner ratio differs from the first by the exact factor
\begin{equation}
 \frac{\binom{4R-2s}{2R-2p+1}}
      {\binom{4R-2s}{2R-2p}}
 =\frac{2R-2q}{2R-2p+1}
 =1+\frac{d-1/2}{R}
 +\frac{(1/2-p)(q+1/2-p)}{R^2}
 +O\!\left(\frac{(1+s)^3}{R^3}\right).
 \label{eq:P-over-Q-local-factor}
\end{equation}
Consequently
\begin{equation}
 \frac{\binom{4R-2s}{2R-2p+1}}{\mathscr C_R}
 =4^{-s}\left(
 1+\frac{\beta_{p,q}}R
 +\frac{\beta^{(2)}_{p,q}}{R^2}
 +O\!\left(\frac{(1+s)^6}{R^3}\right)\right),
 \label{eq:P-corner-second-local}
\end{equation}
where
\begin{equation}
 \beta_{p,q}=\alpha_{p,q}+d-\frac12,
 \label{eq:beta-first-polynomial}
\end{equation}
and
\begin{equation}
 {
 \beta^{(2)}_{p,q}
 =\alpha^{(2)}_{p,q}
 +\alpha_{p,q}\left(d-\frac12\right)
 +\left(\frac12-p\right)\left(q+\frac12-p\right).}
 \label{eq:beta-second-polynomial}
\end{equation}

We now evaluate the two polynomially weighted geometric sums explicitly.
Put
\begin{equation}
 \mu_j:=\sum_{n\ge0}\frac{n^j}{4^n}
 =\left.
 \left(x\frac{d}{dx}\right)^j\frac1{1-x}
 \right|_{x=1/4}.
 \label{eq:quarter-geometric-moments}
\end{equation}
The required values are
\begin{equation}
 \mu_0=\frac43,\quad \mu_1=\frac49,\quad
 \mu_2=\frac{20}{27},\quad \mu_3=\frac{44}{27},\quad
 \mu_4=\frac{380}{81}.
 \label{eq:quarter-moment-values}
\end{equation}
Expanding \eqref{eq:alpha-second-polynomial} in \(p,q\) and collecting
monomials gives
\begin{align}
 \sum_{p,q\ge0}4^{-p-q}\alpha^{(2)}_{p,q}
 ={}&
 \frac14\mu_4\mu_0-\mu_3\mu_1-\frac34\mu_3\mu_0
 +\frac34\mu_2^2+\frac34\mu_2\mu_1\nonumber\\
 &+\frac7{16}\mu_2\mu_0-\frac1{16}\mu_1^2
 -\frac1{16}\mu_1\mu_0
 =\frac{61}{243}.
 \label{eq:alpha-second-geometric-sum}
\end{align}
Likewise \eqref{eq:beta-second-polynomial} gives
\begin{align}
 \sum_{p,q\ge0}4^{-p-q}\beta^{(2)}_{p,q}
 ={}&
 \frac14\mu_4\mu_0-\mu_3\mu_1-\frac34\mu_3\mu_0
 +\frac34\mu_2^2+\frac34\mu_2\mu_1\nonumber\\
 &+\frac{31}{16}\mu_2\mu_0-\frac{25}{16}\mu_1^2
 -\frac{13}{16}\mu_1\mu_0+\frac14\mu_0^2
 =\frac{349}{243}.
 \label{eq:beta-second-geometric-sum}
\end{align}
Equations \eqref{eq:alpha-second-geometric-sum} and
\eqref{eq:beta-second-geometric-sum} prove
\eqref{eq:Q-second-corner}--\eqref{eq:P-second-corner}.

\medskip
\noindent\emph{Internal-centroid contribution.}
We now supply the corner calculation underlying
\eqref{eq:Bint-second-corner}.  Write
\(\mathscr D_L:=\binom{2L}{L}\).  In the exact sum
\eqref{eq:eL-exact}, make the change of variables
\[
 a=L-1-p,\qquad b=L-1-q.
\]
Thus \(0\le q<p\le L-1\), \(d:=p-q\ge1\), \(s:=p+q\), and
\[
 (-1)^a=(-1)^{L-1}(-1)^p,\qquad b-a=d.
\]
The factorial quotient has the exact product form
\begin{align}
 \frac{(a+b+1)!}{(a+1)!(b+1)!\mathscr D_L}
 &=\frac{(2L-1-s)!}{(L-p)!(L-q)!\mathscr D_L}\nonumber\\
 &=\frac{2^{-s-1}}{L}
 \frac{\displaystyle
   \prod_{u=0}^{p-1}\left(1-\frac{u}{L}\right)
   \prod_{v=0}^{q-1}\left(1-\frac{v}{L}\right)}
 {\displaystyle
   \prod_{h=0}^{s}\left(1-\frac{h}{2L}\right)}.
 \label{eq:eL-corner-product}
\end{align}
Logarithmic expansion gives
\begin{equation}
 \frac{(a+b+1)!}{(a+1)!(b+1)!\mathscr D_L}
 =\frac{2^{-s-1}}{L}
 \left(
 1+\frac{A_{p,q}}{L}+\frac{B_{p,q}}{L^2}
 +O\!\left(\frac{(1+s)^6}{L^3}\right)\right),
 \label{eq:eL-corner-expansion}
\end{equation}
uniformly, for example, on \(s\le L^{1/8}\).  The two corner
polynomials are
\begin{align}
 A_{p,q}
 &=-\sum_{u=0}^{p-1}u-\sum_{v=0}^{q-1}v
 +\frac12\sum_{h=0}^{s}h
 =\frac{3s-d^2}{4},
 \label{eq:eL-first-corner-polynomial}\\
 B_{p,q}
 &=\frac12A_{p,q}^2
 -\frac12\sum_{u=0}^{p-1}u^2
 -\frac12\sum_{v=0}^{q-1}v^2
 +\frac18\sum_{h=0}^{s}h^2\nonumber\\
 &=\frac{d^4-10d^2s+4d^2+15s^2-2s}{32}.
 \label{eq:eL-second-corner-polynomial}
\end{align}
Since \(p=q+d\), these become
\begin{align}
 A_{q+d,q}
 &=-\frac{d^2}{4}+\frac{3d}{4}+\frac{3q}{2},
 \label{eq:eL-A-qd}\\
 B_{q+d,q}
 &=\frac{d^4}{32}-\frac{5d^3}{16}
 -\frac{5d^2q}{8}+\frac{19d^2}{32}
 +\frac{15dq}{8}-\frac{d}{16}
 +\frac{15q^2}{8}-\frac{q}{8}.
 \label{eq:eL-B-qd}
\end{align}

Substitution into \eqref{eq:eL-exact} gives
\begin{equation}
 e_L=(-1)^L\frac{\mathscr D_L}{L}
 \left(E_0+\frac{E_1}{L}+\frac{E_2}{L^2}+O(L^{-3})\right),
 \label{eq:eL-three-coefficients-pre}
\end{equation}
where, for \(F=1,A,B\), the relevant linear functional is
\begin{equation}
 \mathcal T[F]
 :=-\frac14\sum_{q\ge0}\sum_{d\ge1}
 d\left(-\frac14\right)^q
 \left(-\frac12\right)^dF(q+d,q),
 \qquad
 (E_0,E_1,E_2)=(\mathcal T[1],\mathcal T[A],\mathcal T[B]).
 \label{eq:eL-corner-functional}
\end{equation}
Put
\[
 \mathsf Q_m:=\sum_{q\ge0}q^m\left(-\frac14\right)^q,
 \qquad
 \mathsf D_m:=\sum_{d\ge1}d^m\left(-\frac12\right)^d.
\]
The needed moments are
\begin{align*}
 \mathsf Q_0&=\frac45,&
 \mathsf Q_1&=-\frac4{25},&
 \mathsf Q_2&=-\frac{12}{125},\\
 \mathsf D_1&=-\frac29,&
 \mathsf D_2&=-\frac2{27},&
 \mathsf D_3&=\frac2{27},&
 \mathsf D_4&=\frac{10}{81},&
 \mathsf D_5&=-\frac{14}{243}.
\end{align*}
Consequently
\[
 E_0=-\frac14\mathsf Q_0\mathsf D_1=\frac2{45},
\]
and, since
\(dA_{q+d,q}=-d^3/4+3d^2/4+3dq/2\),
\begin{equation}
 E_1=-\frac14\left(
 -\frac14\mathsf Q_0\mathsf D_3
 +\frac34\mathsf Q_0\mathsf D_2
 +\frac32\mathsf Q_1\mathsf D_1\right)
 =\frac1{675}.
 \label{eq:eL-E1-evaluation}
\end{equation}
Likewise,
\begin{align*}
 dB_{q+d,q}={}&
 \frac{d^5}{32}-\frac{5d^4}{16}
 -\frac{5d^3q}{8}+\frac{19d^3}{32}
 +\frac{15d^2q}{8}-\frac{d^2}{16}\\
 &+\frac{15dq^2}{8}-\frac{dq}{8},
\end{align*}
so
\begin{align}
 E_2=-\frac14\bigg(&
 \frac1{32}\mathsf Q_0\mathsf D_5
 -\frac5{16}\mathsf Q_0\mathsf D_4
 -\frac58\mathsf Q_1\mathsf D_3
 +\frac{19}{32}\mathsf Q_0\mathsf D_3\nonumber\\
 &+\frac{15}{8}\mathsf Q_1\mathsf D_2
 -\frac1{16}\mathsf Q_0\mathsf D_2
 +\frac{15}{8}\mathsf Q_2\mathsf D_1
 -\frac18\mathsf Q_1\mathsf D_1\bigg)
 =-\frac{109}{6075}.
 \label{eq:eL-E2-evaluation}
\end{align}
Thus
\begin{equation}
 {
 e_L=(-1)^L\frac1L\binom{2L}{L}
 \left(
 \frac2{45}+\frac{1}{675L}
 -\frac{109}{6075L^2}+O(L^{-3})\right).}
 \label{eq:eL-second}
\end{equation}

We next solve the exact renewal relation
\[
 \beta_L=e_L+\sum_{j=1}^{L-1}2^{j-1}e_{L-j}.
\]
Set \(G_L:=(-1)^LL^{-1}\binom{2L}{L}\).  For fixed \(j\), a
factorial-ratio expansion gives
\begin{equation}
 \frac{G_{L-j}}{G_L}
 =(-1)^j4^{-j}
 \left(1+\frac{3j}{2L}
 +\frac{j(15j-1)}{8L^2}
 +O\!\left(\frac{(1+j)^6}{L^3}\right)\right).
 \label{eq:renewal-central-ratio}
\end{equation}
Indeed,
\[
 \frac{\binom{2L-2j}{L-j}}{\binom{2L}{L}}
 =4^{-j}\left(1+\frac{j}{2L}
 +\frac{j(3j-1)}{8L^2}
 +O\!\left(\frac{(1+j)^6}{L^3}\right)\right),
\]
and multiplication by \(L/(L-j)\) gives
\eqref{eq:renewal-central-ratio}.

Write
\[
 c_0=\frac2{45},\qquad c_1=\frac1{675},\qquad
 c_2=-\frac{109}{6075},
\]
and put \(\rho_0:=1\),
\(\rho_j:=(-1)^j2^{-j-1}\) for \(j\ge1\).  Then
\begin{equation}
 W_0:=\sum_{j\ge0}\rho_j=\frac56,
 \qquad
 W_1:=\sum_{j\ge1}j\rho_j=-\frac19,
 \qquad
 W_2:=\sum_{j\ge1}j^2\rho_j=-\frac1{27}.
 \label{eq:renewal-weight-moments}
\end{equation}
Using \((L-j)^{-1}=L^{-1}+jL^{-2}+O((1+j)^2L^{-3})\)
together with \eqref{eq:renewal-central-ratio} gives
\[
 \beta_L=G_L\left(b_0+\frac{b_1}{L}+\frac{b_2}{L^2}
 +O(L^{-3})\right),
\]
where
\begin{align*}
 b_0&=c_0W_0=\frac1{27},\\
 b_1&=c_1W_0+\frac32c_0W_1=-\frac1{162},\\
 b_2&=c_2W_0+\frac52c_1W_1
 +\frac{c_0}{8}(15W_2-W_1)=-\frac{13}{729}.
\end{align*}
Therefore
\begin{equation}
 {
 \beta_L=(-1)^L\frac1L\binom{2L}{L}
 \left(
 \frac1{27}-\frac1{162L}-\frac{13}{729L^2}
 +O(L^{-3})\right).}
 \label{eq:betaL-second}
\end{equation}

Finally, write \(i=R-p\) in
\(\mathfrak B_r^{\rm int}=\sum_{i=0}^{R}\beta_{2i}\).  Uniformly on a
fixed corner window,
\[
 \frac{\binom{4(R-p)}{2(R-p)}}{\mathscr C_R}
 =16^{-p}\left(1+\frac{p}{2R}
 +O\!\left(\frac{(1+p)^2}{R^2}\right)\right).
\]
With \(b_0=1/27\), \(b_1=-1/162\), this yields
\[
 \frac{\beta_{2(R-p)}}{\mathscr C_R}
 =16^{-p}\left[
 \frac{b_0}{2R}+\frac{3pb_0+b_1}{4R^2}
 +O\!\left(\frac{(1+p)^2}{R^3}\right)\right].
\]
Since
\[
 \sum_{p\ge0}16^{-p}=\frac{16}{15},
 \qquad
 \sum_{p\ge0}p16^{-p}=\frac{16}{225},
\]
we obtain
\begin{align*}
 \frac{\mathfrak B_r^{\rm int}}{\mathscr C_R}
 &=\frac{b_0}{2R}\frac{16}{15}
 +\frac1{4R^2}\left(
 3b_0\frac{16}{225}+b_1\frac{16}{15}\right)
 +O(R^{-3})\\
 &=\frac{8}{405R}+\frac{2}{6075R^2}+O(R^{-3}),
\end{align*}
which is \eqref{eq:Bint-second-corner}.

All termwise sums in this internal-centroid calculation are uniform.  In
\eqref{eq:eL-corner-product}, Taylor expansion on
\(s\le L^{1/8}\) has the summable error displayed in
\eqref{eq:eL-corner-expansion}.  On the complement, the standard uniform
binomial estimate and \(\mathscr D_L\asymp4^L/\sqrt L\) reduce the tail,
relative to \(\mathscr D_L/L\), to a polynomially weighted geometric
tail \(\sum_{s>L^{1/8}}s^2 2^{-s}\).  Similarly,
\(|e_n|\ll\binom{2n}{n}/n\), so the range \(j>L^{1/8}\) in the renewal
sum is geometrically small relative to \(G_L\).  Finally,
\(|\beta_{2i}|\ll\binom{4i}{2i}/i\), which gives the same conclusion for
the last \(p\)-corner.  This justifies every expansion through the stated
order.

For the cross term, the factorial ratio in \eqref{eq:B-cross-exact}, with
\(i=R-p\), \(j=R-q\), factors as
\begin{align}
 &\frac{(4R-2p-2q+1)!}
 {(2R-2p+1)!(2R-2q+1)!\mathscr C_R}\nonumber\\
 &\qquad=
 \frac{4^{-p-q}}R\left(
 1+\frac{\gamma_{p,q}}R
 +O\!\left(\frac{(1+p+q)^4}{R^2}\right)\right),
 \label{eq:cross-corner-local-second}
\end{align}
where
\begin{equation}
 {
 \gamma_{p,q}
 =-\frac{p^2+q^2}{2}+pq+\frac{3(p+q)}4-\frac34.}
 \label{eq:gamma-cross-polynomial}
\end{equation}
Indeed, divide the first corner ratio
\eqref{eq:Q-corner-second-local} by the two missing linear factors; the
additional relative coefficient is \((p+q)/2-3/4\).

To sum \eqref{eq:gamma-cross-polynomial}, write \(p=q+d\), \(d\ge1\),
and set \(x=1/4\).  Then
\[
 \gamma_{q+d,q}
 =-\frac{d^2}{2}+\frac{3q}{2}+\frac{3d}{4}-\frac34.
\]
Using \eqref{eq:quarter-geometric-moments} and elementary geometric
summation gives
\begin{equation}
 \sum_{p>q\ge0}(p-q)4^{-p-q}
 =\frac{\mu_1}{1-x^2}=\frac{64}{135}.
 \label{eq:cross-leading-geometric-sum}
\end{equation}
\begin{align}
 \sum_{p>q\ge0}(p-q)4^{-p-q}\gamma_{p,q}
 &=\frac1{1-x^2}
 \left(-\frac12\mu_3+\frac34\mu_2-\frac34\mu_1\right)
 +\frac32\frac{x^2}{(1-x^2)^2}\mu_1\nonumber\\
 &=-\frac{1184}{2025}.
 \label{eq:cross-second-geometric-sum}
\end{align}
This proves \eqref{eq:Bcross-second-corner}; addition gives
\eqref{eq:B-second-corner}.

For completeness, all termwise corner summations above are uniform.  On
the window \(p+q\le R^{1/8}\), Taylor's formula gives the displayed
remainders.  Outside that window,
\[
 \frac{\binom{4R-2p-2q}{\,\cdot\,}}{\mathscr C_R}
 \ll \sqrt R\,4^{-p-q}
\]
by \(\binom Nm\le2^N\) and
\(\mathscr C_R\gg16^R/\sqrt R\); the cross ratio is bounded in the same
way.  Hence the complementary tails are smaller than every fixed inverse
power of \(R\), and the polynomially weighted geometric sums of the local
remainders converge.  This justifies expansion and summation through the
asserted orders.
\end{proof}

\begin{theorem}[Second integrated Edgeworth coefficient]
\label{thm:second-integrated-edgeworth}
The complete raw negative-arch mass has the three-term expansion
\begin{equation}
 \boxed{
 \mathcal A_r
 =\frac{512}{9\sqrt{2\pi}}\,
 \frac{16^r}{\sqrt r}
 \left(
 1-\frac{13}{16r}+\frac{1547}{1536r^2}+O(r^{-3})
 \right).}
 \label{eq:A-second-edgeworth}
\end{equation}
Hence, with $a_0=512/(9\sqrt{2\pi})$,
\begin{equation}
 {
 \mathscr F_{\!A}(z)
 =a_0\operatorname{Li}_{1/2}(z)
 -\frac{13a_0}{16}\operatorname{Li}_{3/2}(z)
 +\frac{1547a_0}{1536}\operatorname{Li}_{5/2}(z)
 +\mathscr E_{A,3}(z),}
 \label{eq:flat-three-polylogs}
\end{equation}
where $[z^r]\mathscr E_{A,3}(z)=O(r^{-7/2})$.  In particular
$\mathscr E_{A,3}$ has two continuous derivatives on $|z|\le1$.
\end{theorem}

\begin{proof}
Insert \cref{lem:second-corner-coefficients} in the exact identity
\eqref{eq:A-QPB-exact}.  The exponentially smaller terms $M_r$ and $R+2$
do not enter any inverse power of $R$.  One obtains
\[
 \mathcal A_r
 =\mathscr C_R\left(
 \frac{32}{9}-\frac{8}{9R}+\frac{14}{27R^2}
 +O(R^{-3})\right).
\]
Using
\[
 \mathscr C_R
 =\frac{16^R}{\sqrt{2\pi R}}
 \left(1-\frac1{16R}+\frac1{512R^2}+O(R^{-3})\right)
\]
and then $R=r+1$ gives \eqref{eq:A-second-edgeworth}; the relative
second coefficient is exactly $1547/1536$.  Termwise summation gives
\eqref{eq:flat-three-polylogs}.
\end{proof}

\section*{Reproducibility statement}

A frozen reproducibility snapshot corresponding to arXiv v1, including
verification scripts, reference outputs, and figures, is archived at Zenodo,
DOI:
\href{https://doi.org/10.5281/zenodo.22250518}
{10.5281/zenodo.22250518}.  The scripts verify the log-periodic Fourier
coefficients, the Gaussian area identity, and the first two integrated
Edgeworth calculations.  Numerical experiments serve only as checks of
formulas proved in the text.  They are not used to infer any theorem.

\section*{Acknowledgments}

The author thanks Beno\^{\i}t Clo\^{\i}tre for his structural work on the
perturbed recursion, in particular the global well-definedness theorem and
the binary arch-and-forest framework on which this paper builds.

\section*{AI-Disclosure}

During the preparation of this manuscript, the author used OpenAI's
ChatGPT for help in language editing, LaTeX restructuring, and the
preparation of verification scripts.  All mathematical statements, proofs,
computations, references, and the final presentation were independently
verified by the author, who assumes full responsibility for the content.

\end{document}